\documentclass[letterpaper,11pt]{amsart}
\usepackage{amsmath}
\usepackage{amssymb}
\usepackage{amsfonts}
\usepackage{amsthm}
\usepackage{enumerate}
\usepackage{tikz-cd}
\usepackage{hyperref}
\usepackage[dvipsnames]{xcolor}
\usepackage{fullpage}
\hypersetup{
    colorlinks, linkcolor=blue,
    citecolor=blue, urlcolor=blue
}

\newcommand{\NN}{\ensuremath{\mathbb{N}}}

\DeclareMathOperator{\Der}{Der}
\DeclareMathOperator{\charr}{char}

\numberwithin{equation}{section}
\newtheorem{theorem}[equation]{Theorem}
\newtheorem{corollary}[equation]{Corollary}
\newtheorem{lemma}[equation]{Lemma}
\newtheorem{proposition}[equation]{Proposition}
\newtheorem{remark}[equation]{Remark}
\theoremstyle{definition}
\newtheorem{definition}[equation]{Definition}
\newtheorem{example}[equation]{Example}
\usepackage[style=numeric, sorting=nyt, backend=biber]{biblatex}
\begin{document}

\title{Graded Differential Polynomial Rings}

\author{Yassine Ait Mohamed}
\address{D\'epartement de math\'ematiques,
Universit\'e de Sherbrooke,
2500 Bd de l'Universit\'e, Sherbrooke, QC, J1K 2R1, Canada}
\email{yassine.ait.mohamed@usherbrooke.ca}

\date{}

\begin{abstract}
We study differential polynomial rings $R[t;\delta]$ over $\Gamma$-graded rings, where $\Gamma$ is an arbitrary group. We show
that $R[t;\delta]$ admits a $\Gamma$-grading compatible with that of $R$ if and only if $\delta$ is a $\gamma$-derivation for some $\gamma\in C_\Gamma(\Gamma_R)$, and that this grading is unique once $\deg(t)=\gamma$ is fixed; if $\delta\neq0$, then $\gamma$ is itself
uniquely determined by $\delta$. We characterize the resulting graded ring by a universal property. We prove a characteristic-free center criterion for gr-simplicity whenever $Z(R[t;\delta])$ is a graded
subring; in characteristic zero, gr-simplicity is equivalent to $\delta$-gr-simplicity of $R$ and $\gamma$-outerness of $\delta$, extending Jordan's simplicity criterion to the graded setting. We further show that $R[t;\delta]$ is gr-prime if and only if $R$ is $\delta$-gr-prime, and that gr-Noetherianity of $R$ passes to $R[t;\delta]$, recovering a graded Hilbert basis theorem as a special case. When $\Gamma$ is abelian, gr-simplicity and gr-primality are shown to be invariants of homogeneous graded Morita equivalence, and every ring homogeneously graded equivalent to $R[t;\delta]$ via a compatible idempotent is again a graded differential polynomial ring.
\end{abstract}
\maketitle

\section{Introduction}
Throughout the paper, all rings are assumed to be associative with identity, and all ring homomorphisms preserve the identity.

Let $R$ be a ring and $\delta$ a derivation of $R$. The differential polynomial ring $R[t;\delta]$ is the
free left $R$-module with basis $\{1, t, t^2, \dotsc\}$, with multiplication defined by the rule
\begin{equation}\label{ORE102}
  tx = xt + \delta(x) \quad \text{for all } x \in R,
\end{equation}
extended by associativity and distributivity \cite{ore1933theory}. Iterating \eqref{ORE102} yields
\begin{equation}\label{ORET103}
  t^n x = \sum_{k=0}^{n} \binom{n}{k} \delta^k(x)\, t^{n-k}
  \quad (x \in R,\ n \geq 0),
\end{equation}
where $\delta^0 = \mathrm{id}$. When $\delta = 0$, this reduces to the ordinary polynomial ring $R[t]$.

 Since ordinary polynomial rings are the special case $\delta=0$, the compatibility problem considered here is related to earlier work on gradings of polynomial rings. In grading frameworks more general than group gradings, including $S$-graded rings inducing $S$, compatibility questions for polynomial rings were considered by Chadeyras \cite{Chadeyras1970} and Krasner \cite{krasner1980anneaux}; see also Kelarev \cite{Kelarev2002} for a systematic treatment of graded ring constructions. Ili\'c-Georgijevi\'c \cite{IlicGeorgijevic2019} extended related constructions to semigroup rings and studied graded radicals of graded polynomial rings. These works provide a natural classical context for the compatibility question studied below.

Differential polynomial rings $R[t;\delta]$ are among the most studied objects in noncommutative algebra \cite{goodearl2004introduction,rowen1988ring}. They arise as special cases of Ore extensions and crossed products, and include enveloping algebras of solvable Lie algebras. In the ungraded setting, chain conditions on $R[t;\delta]$ were studied by Cauchon \cite{cauchon2006ideaux}, and principality was established by Lam \cite{lam2000principal}. Simplicity was first characterized by Jordan \cite{Jordan1975}, who showed that when $\mathrm{char}(R) = 0$, the ring $R[t;\delta]$ is simple if and only if $R$ is $\delta$-simple and $\delta$ is outer; this criterion was extended to arbitrary characteristic by \"{O}inert and Silvestrov \cite{oinert2013maximal}. Primality of $R[t;\delta]$ was studied by Jordan \cite{jordan1975noetherian}, who proved that $R[t;\delta]$ is prime if and only if $R$ is $\delta$-prime. In the graded setting, however, these properties have not been systematically studied for rings graded by an arbitrary group.

The present paper addresses two questions. Given a $\Gamma$-graded ring $R=\bigoplus_{\tau\in  \Gamma_{R}}R_\tau$, where 
$\Gamma_{R}:=\{\tau\in \Gamma : R_\tau\neq 0\}$  denotes the support of the grading, and a derivation  $\delta \in \mathrm{Der}(R)$, under what conditions does $R[t;\delta]$ admit a $\Gamma$-grading compatible with 
that of $R$? And do the classical results on  simplicity, primality, and chain conditions extend to this graded setting?

We assume throughout that $R$ is graded by an arbitrary group $\Gamma$.
\begin{definition}[Good grading]\label{DEFGOODGRADING}
A  $\Gamma$-grading on $R[t;\delta]$   is called a \emph{good grading relative to the given grading of $R$} if  $t$ is homogeneous  , say $\deg(t)=\gamma\in\Gamma$, and the following two conditions hold :
\begin{itemize}
  \item[(i)] the canonical inclusion $\iota \colon R \hookrightarrow R[t;\delta]$ is a graded ring homomorphism;
  \item[(ii)] the   defining relation  \eqref{ORE102} is homogeneous.
\end{itemize}
\end{definition}

 Since a general derivation need not preserve homogeneous components, we restrict to homogeneous derivations, that is, derivations $\delta$ for which there exists a map $\varphi_\delta \colon \Gamma_\delta \to \Gamma$, where \[
\Gamma_\delta := \{\tau \in \Gamma_R : \delta(R_\tau) \neq 0\}
\]denotes the support of $\delta$, satisfying $\delta(R_\tau) \subseteq R_{\varphi_\delta(\tau)}$ for all $\tau \in \Gamma_R$ \cite{GradQ,ait2024,ait2025}. The following example illustrates the range of gradings that arise and motivates the construction developed in this paper.

The first Weyl algebra $A_1(k) = k\langle x, y \rangle/(yx - xy - 1)$ over a field $k$ carries a natural $\mathbb{Z}$-grading with $\deg(x) = 1$ and $\deg(y) = -1$ \cite{won2018noncommutative}, and is isomorphic as a ring to the differential polynomial ring $k[x]\bigl[y\,;\,\tfrac{d}{dx}\bigr]$, where $k[x]$ is graded by $\deg(x) = 1$ \cite{bavula2020classification}. The commutation relation $yr = ry + \tfrac{d}{dx}(r)$ for $r \in k[x]$ reproduces $yx = xy + 1$.
This isomorphism transports the grading of $A_1(k)$ to a $\mathbb{Z}$-grading on $k[x]\bigl[y\,;\,\tfrac{d}{dx}\bigr]$ in which each monomial $x^a y^i$ has degree $a - i$. In particular, $\deg(x) = 1$, $\deg(y) = -1$, and $\tfrac{d}{dx}$ shifts degree by $-1$, that is, $\tfrac{d}{dx}(k[x]_n) \subseteq k[x]_{n-1}$. This suggests considering homogeneous derivations satisfying $\delta(R_\tau) \subseteq R_{\tau\gamma}$ for all $\tau \in \Gamma_R$, for a fixed $\gamma \in C_\Gamma(\Gamma_R)$.
In this case, setting $\deg(t) = \gamma$ and
\begin{equation}\label{ORE130}
  R[t;\delta]_\tau :=
  \bigoplus_{i \geq 0} R_{\tau\gamma^{-i}}\, t^i
  \quad (\tau \in \Gamma)
\end{equation}
yields a good grading on $R[t;\delta]$ (Proposition \ref{PROPGOOD}). Moreover, in the same proposition we show that any homogeneous derivation for which a good grading on $R[t;\delta]$ exists arises in this form.

 We now fix the terminology used in the statements of the main results.

\begin{definition}[Basic graded-ring terminology]\label{DEFGRTERMS}
Let $\Gamma$ be a multiplicatively written group with identity element $e$. A ring $R$ is a \emph{$\Gamma$-graded ring} if \[ R=\bigoplus_{\sigma\in\Gamma}R_\sigma,
\quad R_\sigma R_\tau\subseteq R_{\sigma\tau}, \quad 1_R\in R_e. \]
Its \emph{support} is  $\Gamma_R:=\{\sigma\in\Gamma:R_\sigma\neq 0\},
$ and $\mathcal H(R):=\bigcup_{\sigma\in\Gamma}R_\sigma$ denotes the set of homogeneous elements. We write
\[
C_\Gamma(\Gamma_R):=
\{\gamma\in\Gamma:\gamma\sigma=\sigma\gamma\text{ for every }\sigma\in\Gamma_R\}
\]
for the centralizer of the support and $Z(\Gamma)$ for the center of $\Gamma$.

An ideal $\mathfrak a$ of $R$ is \emph{graded} if
\[\mathfrak a=\bigoplus_{\sigma\in\Gamma}(\mathfrak a\cap R_\sigma).
\]
The ring $R$ is \emph{gr-simple} if it has no nonzero proper graded ideals, \emph{gr-prime} if $\mathfrak a\mathfrak b=0$ for graded ideals $\mathfrak a,\mathfrak b$ implies $\mathfrak a=0$ or $\mathfrak b=0$, and \emph{gr-right-Noetherian} (resp.\ \emph{gr-left-Noetherian}) if every ascending chain of graded right (resp.\ left) ideals stabilizes.

A ring homomorphism $f\colon R\to S$ between $\Gamma$-graded rings is \emph{graded} if $f(R_\sigma)\subseteq S_\sigma$ for every $\sigma\in\Gamma$; a bijective graded homomorphism is a \emph{graded isomorphism}, written $R\cong_{\mathrm{gr}}S$. A \emph{graded field} is a graded ring in which every nonzero homogeneous element is invertible. We write $\operatorname{Aut}_{\mathrm{gr}}(R)$ for the group of graded automorphisms of $R$.
\end{definition}

\begin{definition}[Homogeneous and $\gamma$-derivations]\label{DEFDERIVATIONS}
A derivation $\delta\colon R\to R$ is \emph{homogeneous} if there exists a map \[
\varphi_\delta\colon\Gamma_\delta\to\Gamma,
\quad \Gamma_\delta:=\{\tau\in\Gamma_R:\delta(R_\tau)\neq 0\},
\]
such that $\delta(R_\tau)\subseteq R_{\varphi_\delta(\tau)}$ for every $\tau\in\Gamma_R$ \cite{GradQ,ait2024}. Fix $\gamma\in C_\Gamma(\Gamma_R)$. A homogeneous derivation $\delta$ satisfying
\[\delta(R_\tau)\subseteq R_{\tau\gamma}
\quad(\tau\in\Gamma_R)
\]
is called a \emph{$\gamma$-derivation}, and we write $\delta\in\Der_\gamma(R)$. It is \emph{$\gamma$-inner} if $\delta=\delta_r$ for some $r\in R_\gamma$, where
\[\delta_r(x):=rx-xr,
\]
and \emph{$\gamma$-outer} otherwise.

A graded ideal $\mathfrak a$ of $R$ is \emph{$\delta$-invariant} if $\delta(\mathfrak a)\subseteq\mathfrak a$. The ring $R$ is \emph{$\delta$-gr-simple} if its only $\delta$-invariant graded ideals are $0$ and $R$, and \emph{$\delta$-gr-prime} if $\mathfrak a\mathfrak b=0$ for $\delta$-invariant graded ideals $\mathfrak a,\mathfrak b$ implies $\mathfrak a=0$ or $\mathfrak b=0$.
\end{definition}

\begin{definition}[Graded $R$-rings and the category $\mathcal C_\delta$]\label{DEFCATEGORY}
Let $R$ be a $\Gamma$-graded ring. A \emph{graded $R$-ring} is a pair $(S,\phi)$, where $S$ is a $\Gamma$-graded ring and
\[
\phi\colon R\to S
\]
is a unital graded ring homomorphism; $\phi$ is called the \emph{structure map}. If $(S,\phi)$ and $(S',\phi')$ are graded $R$-rings, a morphism of graded $R$-rings is a graded ring homomorphism $h\colon S\to S'$ satisfying $h\circ\phi=\phi'$.

For $\delta\in\Der_\gamma(R)$, let $\mathcal C_\delta$ be the category whose objects are triples $(S,\phi,u)$ consisting of a graded $R$-ring $(S,\phi)$ and a homogeneous element $u\in S_\gamma$ such that
\begin{equation}\label{EQORE106}
u\phi(x)-\phi(x)u=\phi(\delta(x))
\quad\text{for every }x\in R.
\end{equation}
A morphism
\[
h\colon(S,\phi,u)\to(S',\phi',u')
\]
in $\mathcal C_\delta$ is a morphism of graded $R$-rings satisfying $h(u)=u'$.
\end{definition}

For a subring $A$ of a ring $B$, we use the notation
\[
C_B(A):=\{b\in B:ba=ab\text{ for all }a\in A\}
\]
for the centralizer of $A$ in $B$.

The main results of this paper are as follows.
\begin{theorem}[Universal property]\label{THMORE01}
Let $R$ be a $\Gamma$-graded ring and $\delta\in\Der_\gamma(R)$
with $\gamma\in C_\Gamma(\Gamma_R)$. The pair $(R[t;\delta],t)$ is the
initial object of the category $\mathcal{C}_\delta$. In
particular:
\begin{itemize}
\item[(1)] conjugate $\gamma$-derivations yield isomorphic graded rings;
\item[(2)] an inner shift of $\delta$ by an element of $R_\gamma$ leaves $R[t;\delta]$ unchanged up to graded isomorphism;
\item[(3)] $M_n(R)[t;\hat{\delta}]
      \cong_{\mathrm{gr}} M_n(R[t;\delta])$.
\end{itemize}
\end{theorem}

 \begin{theorem}[Gr-simplicity]\label{THMORE02}
Let $R$ be a $\Gamma$-graded ring, $\gamma\in C_\Gamma(\Gamma_R)$, and $\delta\in\Der_\gamma(R)$.
\begin{itemize}
\item[(1)] If $R$ is gr-simple of characteristic zero and $\delta$ is $\gamma$-outer, then $R[t;\delta]$ is gr-simple.
\item[(2)] If $\charr R = 0$, then $R[t;\delta]$ is gr-simple if and only if $R$ is $\delta$-gr-simple and $\delta$ is $\gamma$-outer.
\item[(3)]  Suppose that $Z(R[t;\delta])$ is a graded subring of $R[t;\delta]$. Then  $R[t;\delta]$ is gr-simple if and only if $R$ is $\delta$-gr-simple and $Z(R[t;\delta])$ is a graded field.
\item[(4)] Gr-simplicity of $R[t;\delta]$ is preserved under conjugation of $\delta$, inner shifts by elements of $R_\gamma$, and matrix extensions.
\item[(5)] If $R$ is gr-simple of characteristic zero, then $R[t;0]\bigl[s;\tfrac{d}{dt}\bigr]$ is gr-simple.
\end{itemize}
\end{theorem}
Part (2) generalizes Jordan's criterion \cite{Jordan1975} to the graded setting. Part (3) removes the characteristic-zero restriction at the cost of a center condition; it is a graded analogue of \cite[Theorem 4.15]{oinert2013maximal}.  The additional hypothesis that $Z(R[t;\delta])$ be graded is automatic, for example, when $\Gamma$ is abelian \cite[Lemma 2.2]{Oinert2026central}, and more generally when $\Gamma$ admits a total order invariant under left and right multiplication and every nonzero homogeneous element of $R[t;\delta]$ is regular (that is, neither a left nor a right zero-divisor) \cite[Proposition 2.1]{Ion1988}.  Part (4) shows that gr-simplicity is stable under the structural operations of Theorem \ref{THMORE01}. Part (5) shows that gr-simplicity persists through iterated differential constructions.

\begin{theorem}[Gr-primality]\label{THM35}
Let $R$ be a $\Gamma$-graded ring, $\gamma\in C_\Gamma(\Gamma_R)$, and $\delta\in\Der_\gamma(R)$.
\begin{itemize}
\item[(1)] If $R$ is gr-prime, then $R[t;\delta]$ is gr-prime.
\item[(2)] $R[t;\delta]$ is gr-prime if and only if $R$ is $\delta$-gr-prime.
\end{itemize}
\end{theorem}

This is a graded analogue of \cite[Theorem 2.2]{jordan1975noetherian}.
\begin{theorem}[Gr-Noetherianity]\label{THMORE03}
Let $\delta\in\Der_\gamma(R)$. If $R$ is gr-right-Noetherian \textup{(}resp.\ gr-left-Noetherian\textup{)}, then so is $R[t;\delta]$. In particular, taking $\delta = 0$ recovers the graded Hilbert basis theorem.
\end{theorem}

 In \cite{abrams2022morita}, Abrams, Ruiz, and Tomforde establish, for rings graded by an \emph{abelian} group, the equivalence of four conditions (HG1)--(HG4) constituting the Homogeneously Graded Extended Morita Theorem \cite[Theorem~4.7]{abrams2022morita}. Two of these conditions are purely algebraic:
\begin{align}
  &A \cong_{\mathrm{gr}} \ell M_n(B)\ell \tag{C2}\label{HG2} \\
  &M_\infty(A) \cong_{\mathrm{gr}} M_\infty(B), \tag{C4}\label{HG4}
\end{align}
for some $n\in\mathbb{N}$ and full homogeneous idempotent
$\ell\in M_n(B)_e$, where all matrix rings carry the \emph{standard grading} $(M_X(A))_\sigma := M_X(A_\sigma)$. The equivalence \eqref{HG2}$\Leftrightarrow$\eqref{HG4} is proved in \cite[Theorem~2.14]{abrams2022morita} through degree-$e$ idempotents and matrix operations assembled via direct limits; the argument never invokes commutativity of $\Gamma$. The abelian hypothesis of \cite{abrams2022morita} is therefore a global convention for the categorical parts of their theorem and plays no role in the proof of \eqref{HG2}$\Leftrightarrow$\eqref{HG4}. E.~Ruiz has indicated that the full Theorem 4.7 of \cite{abrams2022morita}, including the categorical conditions, is expected to extend to arbitrary groups (E. Ruiz, private communication, 2026). Since a proof of this extension is not yet available, we work throughout with $\Gamma$ abelian and recall the following definition.

\begin{definition}\label{GORE100}
Let $\Gamma$ be an abelian group and let $A$, $B$ be $\Gamma$-graded
rings.  We say $A$ and $B$ are \emph{homogeneously
graded equivalent}, written $A \approx_{\mathrm{hgr}} B$, if there exist
$n \in \mathbb{N}$ and a full homogeneous idempotent $\ell \in M_n(B)_e$
such that $A \cong_{\mathrm{gr}} \ell M_n(B)\ell$, or equivalently
$M_\infty(A) \cong_{\mathrm{gr}} M_\infty(B)$
\cite[Theorem~2.14]{abrams2022morita}.
\end{definition}
Under the hypothesis that $\Gamma$ is abelian, we prove the following results.

\begin{theorem}[Morita invariance]\label{THMOR345}
Let $R$ be a graded ring and  $\delta\in\Der_\gamma(R)$.
\begin{itemize}
\item[(1)] If $A \approx_{\mathrm{hgr}} R[t;\delta]$, then $A$ is gr-simple \textup{(}resp.\ gr-prime\textup{)} whenever $R[t;\delta]$ is.
\item[(2)] If $A \approx_{\mathrm{hgr}} R[t;\delta]$ arises from a full idempotent $\ell\in R_e$ with $\delta(\ell)=0$ \textup{(}so that $\delta$ restricts to a
      $\gamma$-derivation on $\ell R\ell$\textup{)}, then
      $A \cong_{\mathrm{gr}} (\ell R\ell)[t;\bar{\delta}]$,
      where $\bar{\delta}(s):=\ell\delta(s)\ell$.
\end{itemize}
\end{theorem}
Part (2) shows that the differential polynomial structure is itself an invariant of $\approx_{\mathrm{hgr}}$: every ring homogeneously graded equivalent to $R[t;\delta]$ via a compatible idempotent is again a graded differential polynomial ring.

\medskip

 The paper is organized as follows. In \S\ref{SECORE104}, we construct the graded differential polynomial ring $R[t;\delta]$. We characterize the derivations for which a good grading exists (Proposition \ref{PROPGOOD}), establish the universal property of this construction (Theorem \ref{THMORE105}), and derive the structural isomorphisms given in Corollaries \ref{CORORE107}--\ref{COROORE109}. In \S\ref{SECORE111}, we establish the graded ring-theoretic properties of $R[t;\delta]$: we prove gr-simplicity (Theorems \ref{THMORE02}(1)--(5)), gr-primality (Theorem \ref{THM35}; see also Section \ref{SECORE125}), and gr-Noetherianity (Theorem \ref{THMORE03}). In \S\ref{SECMORITA}, we treat the case where the grading group is abelian. We prove that gr-simplicity and gr-primality are invariants of homogeneous graded Morita equivalence (Proposition \ref{PORPOREMORITA123}), and we show that the differential polynomial ring structure itself is preserved under equivalence via a compatible idempotent (Proposition \ref{POPDP2563t}, Theorem \ref{THMOR345}).

\section{Construction of graded differential polynomial rings}\label{SECORE104}
We begin by constructing the graded differential polynomial ring $R[t;\delta]$ when the derivation shifts degrees by a fixed element of   the centralizer of the  support of the grading . We describe the grading on $R[t;\delta]$, explain the universal mapping property that characterizes the construction, and introduce the leading-coefficient ideals that will serve as the main technical tool.

\begin{proposition}\label{PROPGOOD}
Let $\gamma \in C_\Gamma(\Gamma_R)$ and define $R[t;\delta]_\tau$
by \eqref{ORE130} with $\deg(t) = \gamma$. 
\begin{itemize}
\item[(1)] If $\delta$ is a $\gamma$-derivation of $R$, then \eqref{ORE130} defines a good grading on $R[t;\delta]$.
\item[(2)] If $\delta$ is a homogeneous derivation such that there exists a good $\Gamma$-grading on $R[t;\delta]$, then $\delta$ is a $\gamma$-derivation.
\end{itemize}
\end{proposition}

\begin{proof}
$(1)$ Assume $\gamma \in C_\Gamma(\Gamma_{R})$ and  $\delta \in \Der_\gamma(R)$. We first show that  $R[t;\delta] = \bigoplus_{\tau \in \Gamma}  R[t;\delta]_\tau$. Since $R[t;\delta]$ is  the free left $R$-module with basis $\{1, t, t^2, \dotsc\}$  and $R = \bigoplus_{\sigma \in \Gamma_R}  R_\sigma$, every element of $R[t;\delta]$ has  a unique expression $\sum_{i \geq 0} r_i\, t^i$ with  $r_i \in R$, and each $r_i$ decomposes uniquely as  $r_i = \sum_{\sigma \in \Gamma_R} (r_i)_\sigma$ with  $(r_i)_\sigma \in R_\sigma$. Setting $\tau := \sigma\gamma^i$, 
the term $(r_i)_{\tau\gamma^{-i}}\,t^i$ belongs to 
$R_{\tau\gamma^{-i}}\,t^i \subseteq 
R[t;\delta]_\tau$, so $R[t;\delta] =  \sum_{\tau \in \Gamma} R[t;\delta]_\tau$. To see the sum  is direct, suppose $\sum_{\tau} f_\tau = 0$ with $f_\tau = 
\sum_{i \geq 0} c_{\tau,i}\,t^i \in R[t;\delta]_\tau$ 
and $c_{\tau,i} \in R_{\tau\gamma^{-i}}$. The coefficient 
of $t^i$ gives $\sum_\tau c_{\tau,i} = 0$ in $R$. Since distinct values of $\tau$ correspond to distinct homogeneous
components $R_{\tau\gamma^{-i}}$ of $R$, the direct-sum decomposition of $R$ forces $c_{\tau,i}=0$ for all $\tau$ and $i$, hence  $f_\tau = 0$ for all $\tau$. Therefore \[R[t;\delta] 
  = \bigoplus_{\tau \in \Gamma} R[t;\delta]_\tau.
\]

 It remains to verify conditions (i) and (ii). For (i), if  $x \in R_\tau$ then $\iota(x) = x\,t^0 \in 
R_{\tau\gamma^0}\,t^0 = R_\tau\,t^0 \subseteq 
R[t;\delta]_\tau$, so $\iota$ is graded. For~(ii), let  $f = x\,t^i \in R[t;\delta]_\tau$ with $x \in 
R_{\tau\gamma^{-i}}$ and $g = x'\,t^j \in 
R[t;\delta]_\sigma$ with $x' \in 
R_{\sigma\gamma^{-j}}$. By \eqref{ORET103},
\[
  fg = \sum_{k=0}^{i}\binom{i}{k} 
  x\,\delta^k(x')\,t^{i+j-k}.
\]

Since $\delta \in \Der_\gamma(R)$, we have 
$\delta^k(x') \in R_{\sigma\gamma^{k-j}}$, and since 
$\gamma \in C_\Gamma(\Gamma_{R})$,
\[
  x\,\delta^k(x') \in 
  R_{\tau\gamma^{-i}}\,R_{\sigma\gamma^{k-j}} 
  \subseteq R_{\tau\sigma\gamma^{-(i+j-k)}},
\] which is precisely the $(\tau\sigma)$-component of \eqref{ORE130}  at index $i+j-k$. Hence $fg \in R[t;\delta]_{\tau\sigma}$, so the grading is multiplicative. In particular, for $x \in  R_\tau$, both sides of \eqref{ORE102} are homogeneous of  degree $\tau\gamma$: the left-hand side $tx$ lies in  $R[t;\delta]_{\gamma\tau} = 
R[t;\delta]_{\tau\gamma}$ since 
$\gamma \in C_\Gamma(\Gamma_{R})$, the term $xt$ lies in  $R[t;\delta]_{\tau\gamma}$, and $\delta(x) \in  R_{\tau\gamma} \subseteq R[t;\delta]_{\tau\gamma}$ 
since $\delta \in \Der_\gamma(R)$. Hence the 
relation \eqref{ORE102} is homogeneous of degree $\tau\gamma$ for  every $\tau \in \Gamma_R$, and \eqref{ORE130} is a good 
grading.

\medskip

$(2)$ Suppose $\delta$ is a homogeneous derivation and that
$R[t;\delta]$ admits a good $\Gamma$-grading with
$\deg(t) = \gamma$. Let $\tau \in \Gamma_R$ and
$x \in R_\tau$ be nonzero. By condition (i), $x \in
R[t;\delta]_\tau$, and since $t \in
R[t;\delta]_\gamma$, the grading axiom gives
\[tx \in R[t;\delta]_{\gamma\tau}, \quad
  xt \in R[t;\delta]_{\tau\gamma}.
\]
By condition $(ii)$, the defining relation is homogeneous of degree $\gamma\tau$, hence $xt$ and $\delta(x)$ must both lie in the homogeneous component containing $tx$. Since $\tau \in \Gamma_R$ was arbitrary, $\gamma \in C_\Gamma(\Gamma_R)$. Moreover, all three terms have degree $\tau\gamma$, so $\delta(x) \in R[t;\delta]_{\tau\gamma}$. Since $\delta(x) \in R$ and $\iota$ is graded, $R \cap R[t;\delta]_{\tau\gamma} =
R_{\tau\gamma}$, hence $\delta(x) \in
R_{\tau\gamma}$. Since $\tau \in \Gamma_R$
and $x \in R_\tau$ were arbitrary,
$\delta \in \Der_\gamma(R)$.
\end{proof}

The preceding proposition determines when a good grading exists. Once the degree of $t$ is fixed, this grading is unique.

\begin{proposition}\label{PROPUNIQUEGRADING}
Let $\gamma\in C_\Gamma(\Gamma_R)$ and
$\delta\in\Der_\gamma(R)$. The grading \eqref{ORE130} is the unique $\Gamma$-grading on $R[t;\delta]$ for which the canonical inclusion
$\iota\colon R\hookrightarrow R[t;\delta]$ is graded and $t$ is homogeneous of degree $\gamma$. If $\delta\neq0$, then $\gamma$ is
uniquely determined by $\delta$.
\end{proposition}

\begin{proof}
Let
\[
R[t;\delta]=\bigoplus_{\tau\in\Gamma}A_\tau
\]
be such a grading, and set
$B_\tau:=\bigoplus_{i\ge0}R_{\tau\gamma^{-i}}t^i$.
If $x_i\in R_{\tau\gamma^{-i}}$, then
$x_i\in A_{\tau\gamma^{-i}}$ and $t^i\in A_{\gamma^i}$, hence
$x_it^i\in A_\tau$. Thus $B_\tau\subseteq A_\tau$.

Conversely, let $f\in A_\tau$ and write
$f=\sum_{i\ge0}x_it^i$, with
$x_i=\sum_{\sigma\in\Gamma}(x_i)_\sigma$.
Each $(x_i)_\sigma t^i$ is homogeneous of degree
$\sigma\gamma^i$. Grouping terms of equal degree gives
$f=\sum_{\eta\in\Gamma}g_\eta$ with $g_\eta\in A_\eta$.
Since $f\in A_\tau$, uniqueness of the homogeneous decomposition
gives $g_\eta=0$ for $\eta\neq\tau$. Therefore
\[
f=g_\tau
=\sum_{i\ge0}(x_i)_{\tau\gamma^{-i}}t^i
\in B_\tau.
\]
Hence $A_\tau=B_\tau$ for every $\tau\in\Gamma$.

Finally, suppose $\delta\neq0$ and
$\delta\in\Der_{\gamma'}(R)$ for some
$\gamma'\in C_\Gamma(\Gamma_R)$. Choose
$x\in R_\sigma$ with $\delta(x)\neq0$. Then
$\delta(x)\in R_{\sigma\gamma}\cap R_{\sigma\gamma'}$.
Since this intersection contains a nonzero homogeneous element,
$\sigma\gamma=\sigma\gamma'$, and therefore
$\gamma=\gamma'$.
\end{proof}

\begin{example}\label{EX0123}
Let $k$ be a field, $\Gamma = S_3 = \langle \sigma, \eta \mid
\sigma^3 = \eta^2 = e,\ \eta\sigma\eta = \sigma^{-1}\rangle$, and
$R = kS_3$ the group algebra with its natural
$S_3$-grading $R_\tau = k\tau$. Then
$\Gamma_{R} = S_3$ and
$C_{S_3}(\Gamma_{R})= \{e\}$.
Take $\gamma = \eta \notin C_{S_3}(\Gamma_{R})$ and
$\delta = 0$. Suppose for a contradiction that
$R[t;0]$ admits a good grading with $\deg(t) = \eta$.
Take $x = \sigma \in R_\sigma$. By condition~(i),
$x \in R[t;0]_\sigma$, and since
$t \in R[t;0]_\eta$, the grading axiom gives
\[
 tx \in R[t;0]_{\eta\sigma}, \quad
  xt \in R[t;0]_{\sigma\eta}.
\] Condition (ii) requires the relation $tx = xt$ to be homogeneous, that is, both sides must have the same degree. But $\eta\sigma \neq \sigma\eta$ in $S_3$, so $tx$ and $xt$ have distinct degrees, contradicting condition (ii). Hence no good grading exists.
\end{example}

 The preceding example uses the zero derivation. The next  shows that the centralizer condition cannot be dropped even for a nonzero homogeneous derivation.

\begin{example}[A nonzero derivation and the centralizer condition]\label{EXNONZEROCENTRALIZER}
Let $\mathbb F$ be a field and let $\Gamma=S_3=\langle \sigma,\eta\mid \sigma^3=\eta^2=e,\ \eta\sigma\eta=\sigma^{-1}\rangle$. Let $V=\mathbb F v\oplus \mathbb F w$ be a two-dimensional graded $\mathbb F$-vector space with
 $\deg(v)=\sigma,
\quad
\deg(w)=\sigma\eta,
$
and let $R:=\mathbb F\oplus V$ be the square-zero extension, so $V^2=0$, with
\[
R_e=\mathbb F,\quad R_\sigma=\mathbb  F v,\quad R_{\sigma\eta}=\mathbb F w,
\]
and all other homogeneous components zero. Define the $\mathbb F $-linear map $\delta\colon R\to R$ by
\[
\delta(1)=0,\quad \delta(v)=w,\quad \delta(w)=0.
\]
Then $\delta$ is a nonzero $\eta$-derivation, since $\delta(R_\tau)\subseteq R_{\tau\eta}$ for every $\tau\in\Gamma_R$. However, $\eta\notin C_\Gamma(\Gamma_R)$ because $\eta\sigma\neq\sigma\eta$.

Suppose that $R[t;\delta]$ admitted a good grading with $\deg(t)=\gamma'$. Since $v$ and $w$ are homogeneous, the relation
$tv-vt=\delta(v)=w$ forces $\gamma'\sigma=\sigma\gamma'$ and $\gamma'\sigma=\sigma\eta$. Hence $\gamma'=\eta$, and therefore $\eta\sigma=\sigma\eta$, a contradiction. Thus $R[t;\delta]$ admits no good grading.
\end{example}

The following remark gives another interpretation of the centralizer condition in Proposition \ref{PROPGOOD}.

\begin{remark}\label{REMSQUAREZERO}
For $\gamma\in\Gamma$, let $R(\gamma)$ denote the additive group $R$
with shifted grading
\[
R(\gamma)_\sigma:=R_{\sigma\gamma}.
\]
We equip $R(\gamma)$ with the left and right $R$-actions induced by
multiplication in $R$. The left action is always graded, since
$R_\sigma R(\gamma)_\tau\subseteq R(\gamma)_{\sigma\tau}$.
The right action is graded if and only if
$\gamma\in C_\Gamma(\Gamma_R)$.

Indeed, if $\gamma$ centralizes $\Gamma_R$, then
$R(\gamma)_\sigma R_\tau\subseteq R(\gamma)_{\sigma\tau}$.
Conversely, suppose that the right action is graded and let
$\tau\in\Gamma_R$. Since
$1_R\in R(\gamma)_{\gamma^{-1}}=R_e$, we have
\[
R_\tau
=1_RR_\tau
\subseteq R(\gamma)_{\gamma^{-1}\tau}
=R_{\gamma^{-1}\tau\gamma}.
\]
As $R_\tau\neq0$, it follows that
$\gamma^{-1}\tau\gamma=\tau$. Thus
$\gamma\in C_\Gamma(\Gamma_R)$.

Suppose now that this condition holds, and form the square-zero extension
\[
T:=R\ltimes R(\gamma),\quad
(r,m)(r',m')=(rr',rm'+mr'),
\]
with $T_\sigma:=R_\sigma\oplus R(\gamma)_\sigma$.
Then $T$ is a $\Gamma$-graded ring.

If $\delta\colon R\to R$ is additive, define
$\varphi_\delta(x):=(x,\delta(x))$. A direct computation shows that $\varphi_\delta\colon R\to T$ is a ring homomorphism if and only if $\delta$ is a derivation. Moreover, $\varphi_\delta$ is graded if and only if $\delta(R_\sigma)\subseteq R_{\sigma\gamma}$ for every $\sigma\in\Gamma_R$. Hence the $\gamma$-derivations of $R$ are
naturally identified with the graded ring sections of the projection $T\to R$, $(r,m)\mapsto r$.
\end{remark}

 We now establish the universal property of $R[t;\delta]$, which underlies the structural isomorphisms of Corollaries \ref{CORORE107}--\ref{COROORE109}.

\begin{theorem}[Universal property]\label{THMORE105}
Let $R$ be a graded ring and $\delta \in \Der_\gamma(R)$. The pair $(R[t;\delta], t)$ is the initial object of $\mathcal{C}_\delta$: for every graded ring $S$, graded ring homomorphism $\phi \colon R \to S$, and homogeneous element $u \in S_\gamma$ satisfying \eqref{EQORE106}, there exists a unique graded ring homomorphism $\widehat{\phi} \colon R[t;\delta] \to S$ making the diagram
\begin{equation*}
\begin{tikzcd}[column sep=large, row sep=large]
  R \arrow[r, hook, "\iota"] \arrow[dr, "\phi"']
  & R[t;\delta] \arrow[d, dashed, "\exists!\,\widehat{\phi}"] \\
  & S
\end{tikzcd}
\end{equation*}
commute and satisfying $\widehat{\phi}(t) = u$.
\end{theorem}

\begin{proof}
Any ring homomorphism $R[t;\delta] \to S$ extending $\phi$ with $t \mapsto u$ must satisfy $\psi(x_i t^i) = \phi(x_i)\,u^i$ for all $x_i \in R$ and $i \geq 0$, so there is at most one such map, namely
\[\widehat{\phi}\Bigl(\sum_{i \geq 0} x_i t^i\Bigr) :=
  \sum_{i \geq 0} \phi(x_i)\, u^i.
\]
This formula defines an additive map. Since $R[t;\delta]$ is
generated by $R$ and $t$ subject only to \eqref{ORE102},
multiplicativity reduces to verifying compatibility with that
relation. For any $x \in R$,
\[
  \widehat{\phi}(tx)
  = u\,\phi(x)
  = \phi(x)\,u + \phi(\delta(x))
  = \widehat{\phi}(xt + \delta(x)),
\]
where the middle equality is \eqref{EQORE106}, so $\widehat{\phi}$
is a ring homomorphism. It is graded: for $x_i \in R_{\tau\gamma^{-i}}$, one has $\phi(x_i) \in S_{\tau\gamma^{-i}}$ and $u^i \in S_{\gamma^i}$,
hence $\phi(x_i)\,u^i \in S_\tau$.
\end{proof}

The universal property yields the following structural isomorphisms, which together prove Theorem \ref{THMORE01}.

\begin{corollary}\label{CORORE107}
Let $\delta_1,\delta_2 \in \Der_\gamma(R)$ and suppose
$\delta_2 = \alpha\circ\delta_1\circ\alpha^{-1}$ for some $\alpha\in \mathrm{Aut}_{\mathrm{gr}}(R)$. There is a graded isomorphism $\widetilde{\alpha} \colon R[t;\delta_1] \xrightarrow{\;\sim\;} R[t;\delta_2]$ extending $\alpha$.
\end{corollary}

\begin{proof}
For any $x \in R$, the commutation rule in $R[t;\delta_2]$ gives
$t\,\alpha(x) - \alpha(x)\,t = \alpha(\delta_1(x))$, so
\eqref{EQORE106} holds in $R[t;\delta_2]$ with $\phi = \alpha$ and $u = t$. Theorem \ref{THMORE105} produces a unique graded ring homomorphism $\widetilde{\alpha} \colon R[t;\delta_1] \to
R[t;\delta_2]$ extending $\alpha$ with $\widetilde{\alpha}(t) = t$. Applying the same argument to $\alpha^{-1}$ yields an inverse, so $\widetilde{\alpha}$ is an isomorphism.
\end{proof}

\begin{corollary}\label{CORORE1078}
Let $\delta = \delta_1 + \delta_2$ where
$\delta_1, \delta_2 \in \Der_\gamma(R)$ and
$\delta_1 = \delta_r$ is the inner derivation
$\delta_r(x) := rx - xr$ for some $r \in R_\gamma$. Then
$R[t;\delta] \cong_{\mathrm{gr}} R[t;\delta_2]$.
\end{corollary}

\begin{proof}
Set $u := t - r \in R[t;\delta]$. Since $r \in R_\gamma$ and $t \in R[t;\delta]_\gamma$, the element $u$ is homogeneous of
degree $\gamma$. For any $x\in R$, \[ux - xu = (t-r)x - x(t-r) = tx - xt - (rx - xr) = \delta(x) - \delta_r(x) = \delta_2(x),
\] so $(R[t;\delta], u)$ is an object of $\mathcal{C}_{\delta_2}$. By Theorem \ref{THMORE105}, there is a unique graded ring
homomorphism $\varphi \colon R[t;\delta_2] \to R[t;\delta]$ with $\varphi_{|R} = \mathrm{id}$ and $\varphi(t) = t - r$.

For the inverse, set $v := t + r \in R[t;\delta_2]$. Since $\delta_r(x) = rx - xr$, a direct computation gives $vx - xv = \delta(x)$, so $(R[t;\delta_2], v)$ is an object of $\mathcal{C}_\delta$. Theorem \ref{THMORE105} gives a graded ring homomorphism $\psi \colon R[t;\delta] \to R[t;\delta_2]$ with $\psi_{|R} = \mathrm{id}$ and $\psi(t) = t + r$. Both $\psi \circ \varphi$ and $\varphi \circ \psi$ fix $R$ pointwise and fix $t$, so uniqueness in Theorem \ref{THMORE105} gives $\psi \circ \varphi = \mathrm{id}$ and $\varphi \circ \psi = \mathrm{id}$.
\end{proof}

\begin{corollary}\label{COROORE109}
Let $T$ be a graded ring, $\delta \in \Der_\gamma(T)$, and
$R := M_n(T)$ with the grading $M_n(T)_\tau := M_n(T_\tau)$. Let
$\hat{\delta}$ be the entrywise $\gamma$-derivation
$\hat{\delta}\bigl((x_{ij})\bigr) := \bigl(\delta(x_{ij})\bigr)$.
Then $M_n(T)[t;\hat{\delta}] \cong_{\mathrm{gr}} M_n(T[t;\delta])$.
\end{corollary}

\begin{proof}
Let $\phi \colon M_n(T) \to M_n(T[t;\delta])$ be the map sending
$(x_{ij})$ to the same matrix with entries in $T[t;\delta]$. This is a graded ring homomorphism since $T \hookrightarrow T[t;\delta]$ is graded. Set $u := tI_n \in M_n(T[t;\delta])_\gamma$. The relation $tx_{ij} - x_{ij}t = \delta(x_{ij})$ in $T[t;\delta]$ gives,
entrywise,
\[
  u\,\phi(A) - \phi(A)\,u
  = \bigl(\delta(x_{ij})\bigr)_{ij}
  = \phi(\hat{\delta}(A))
\]
for any $A = (x_{ij}) \in M_n(T)$, so \eqref{EQORE106} holds and
Theorem \ref{THMORE105} yields a unique graded ring homomorphism
$\varphi \colon M_n(T)[t;\hat{\delta}] \to M_n(T[t;\delta])$
with $\varphi_{|M_n(T)} = \phi$ and $\varphi(t) = tI_n$.

Both rings are free left $M_n(T)$-modules on the basis $\{I_n, t, t^2, \ldots\}$, and $\varphi$ acts by $\varphi\bigl(\sum_k A_k t^k\bigr) = \sum_k A_k t^k$. Surjectivity is immediate. If $\sum_k A_k t^k \in \ker\varphi$, then each $A_k = 0$ by uniqueness of coefficients, so $\varphi$ is injective.
\end{proof}

For a  nonzero element \[
f=\sum_{i=0}^{n}x_i t^i\in R[t;\delta],
\quad x_n\neq 0,
\]
we write $\deg_t(f):=n$ for its $t$-degree and $\operatorname{lc}(f):=x_n$ for its leading coefficient. For a  graded ideal $\mathfrak{a} \subseteq R[t;\delta]$ and $n \geq 0$, the \emph{$n$-th leading-coefficient ideal} is
\[
  \mathcal{L}_n(\mathfrak{a})
  := \{0\}
  \cup \bigl\{x \in R :
    x \text{ is the leading coefficient of some element of }
    \mathfrak{a} \text{ of } t\text{-degree } n
  \bigr\}.
\]

\begin{lemma}\label{LEMORE112}
Let $\mathfrak{a} \subseteq R[t;\delta]$ be a graded ideal and $n \geq 0$. Then $\mathcal{L}_n(\mathfrak{a})$ is a $\delta$-invariant graded ideal of $R$.  
\end{lemma}

\begin{proof}
We first show that $\mathcal{L}_n(\mathfrak{a})$ is graded.
Let $x \in \mathcal{L}_n(\mathfrak{a})$, witnessed by
$f = \sum_{i=0}^{n} a_i\, t^i \in \mathfrak{a}$ with $a_n = x$. Write $a_i = \sum_{\sigma  \in \Gamma_R}(a_i)_\sigma$ for the homogeneous decomposition of each coefficient in $R$. By \eqref{ORE130}, the $\tau$-homogeneous component of $f$ in $R[t;\delta]$ is
\begin{equation}\label{EQ:FTAU}
  f_\tau = \sum_{i=0}^{n} (a_i)_{\tau\gamma^{-i}}\, t^i.
\end{equation}
Since $\mathfrak{a}$ is a graded ideal, $f_\tau \in \mathfrak{a}$ for
every $\tau \in \Gamma$. Let  $\sigma \in \Gamma_R$ with $x_\sigma \neq 0$; setting $\tau = \sigma\gamma^n$ in \eqref{EQ:FTAU} shows that
$f_{\sigma\gamma^n} \in \mathfrak{a}$
has $t$-degree $n$ with leading coefficient
$(a_n)_\sigma = x_\sigma$. Hence $x_\sigma \in \mathcal{L}_n(\mathfrak{a})$. Every homogeneous component of $x$ thus lies in $\mathcal{L}_n(\mathfrak{a})$,
giving
\[
  \mathcal{L}_n(\mathfrak{a})
  = \bigoplus_{\sigma \in \Gamma_R}
    \bigl(\mathcal{L}_n(\mathfrak{a}) \cap R_\sigma\bigr),
\]
so $\mathcal{L}_n(\mathfrak{a})$ is graded.

We now verify the ideal property on homogeneous elements,
which suffices since both $R$ and $\mathcal{L}_n(\mathfrak{a})$ are graded. Let $x \in \mathcal{L}_n(\mathfrak{a}) \cap R_\sigma$, witnessed by $f = xt^n + \cdots \in \mathfrak{a}_{\sigma\gamma^n}$,
and let $r \in R_\tau$. For left multiplication, $rf \in \mathfrak{a}_{\tau\sigma\gamma^n}$ has $t$-degree $n$ with leading coefficient $rx \in R_{\tau\sigma}$, so $rx \in \mathcal{L}_n(\mathfrak{a})$.
For right multiplication, \eqref{ORET103} gives
\[
  fr = xr\, t^n + (\text{terms of $t$-degree} < n);
\]
since $\gamma \in C_\Gamma(\Gamma_R)$, the product $fr$ is homogeneous of degree $\sigma\tau\gamma^n$, so $fr \in \mathfrak{a}_{\sigma\tau\gamma^n}$ and $xr \in \mathcal{L}_n(\mathfrak{a})$.

For $\delta$-invariance, since $\mathfrak{a}$ is a two-sided ideal, both $tf$ and $ft$ lie in $\mathfrak{a}$, and a direct computation gives
\[tf - ft
  = \bigl[xt^{n+1} + (\delta(x) + a_{n-1})t^n + \cdots\bigr]
  - \bigl[xt^{n+1} + x_{n-1}t^n + \cdots\bigr]
  =\delta(x)\,t^n + \cdots,
\]
so $tf - ft \in \mathfrak{a}$ has $t^n$-coefficient $\delta(x)$, whence $\delta(x) \in \mathcal{L}_n(\mathfrak{a})$.
\end{proof}

\begin{lemma} \label{LEM_QALGEBRA}
Let $R$ be a $\delta$-gr-simple ring of characteristic zero. Then $R$ is torsion-free as a $\mathbb Z$-module and every integer $n\geq 1$ is invertible in $R$. In particular, the canonical map $\mathbb Z\to R$ extends uniquely to an embedding $\mathbb Q\hookrightarrow Z(R)$, so $R$ is a $\mathbb Q$-algebra.
\end{lemma}

\begin{proof}
Let \[
\mathfrak t:=\{x\in R:\text{there exists }n\geq 1\text{ with }nx=0\}.
\]
The set $\mathfrak t$ is a two-sided graded ideal, and it is $\delta$-invariant because $n\delta(x)=\delta(nx)=0$. Since $\charr R=0$, one has $1_R\notin\mathfrak t$, so $\delta$-gr-simplicity gives $\mathfrak t=0$. Thus $R$ is torsion-free as a $\mathbb Z$-module.

Now fix $n\geq 1$. The ideal
\[
nR=(n\cdot 1_R)R
\]
is nonzero, graded, two-sided, and $\delta$-invariant, since $n\cdot1_R\in R_e\cap Z(R)$ and $\delta(nr)=n\delta(r)$. Hence $nR=R$. Thus $n\cdot1_R$ is invertible in $R$. Since it is central, its inverse is also central, and the stated embedding of $\mathbb Q$ follows.
\end{proof}

\begin{lemma} \label{LEMORE110}
Let $R$ be a gr-simple ring of characteristic zero and  $\delta \in \Der_\gamma(R)$ a $\gamma$-outer derivation. If  $s = \sum_{i=0}^{n} x_i t^i \in C_{R[t;\delta]}(R)$ and there exists  $\sigma \in Z(\Gamma)$ with $(x_n)_\sigma \neq 0$, then $n = 0$  and $s \in Z(R)$.
\end{lemma}

\begin{proof}
Suppose $n > 0$ and  let $c \in \mathcal{H}(R)$.  Using \eqref{ORET103},
\[
  cs = \sum_{i=0}^n cx_i\,t^i, \quad
  sc = \sum_{i=0}^n x_i t^i c
     = \sum_{i=0}^n \sum_{k=0}^{i}\binom{i}{k}x_i\,\delta^k(c)\,t^{i-k}.
\]
Comparing the $t^n$-coefficients of $cs = sc$ gives $cx_n = x_nc$ for all
homogeneous $c \in R$, so $x_n \in Z(R)$. Since $\sigma \in Z(\Gamma)$,
it follows that $(x_n)_\sigma R$
is a graded ideal of $R$; gr-simplicity then forces $(x_n)_\sigma \in R^\times$. Comparing the $t^{n-1}$-coefficients of $cs = sc$ gives \begin{equation}\label{EQ:TN1}
  cx_{n-1} - x_{n-1}c = n\,x_n\,\delta(c) \quad \text{for all  }
  c \in \mathcal{H}(R).
\end{equation}
For $c \in R_\tau$, we have $\delta(c) \in R_{\tau\gamma}$.  Taking the $R_{\sigma\tau\gamma}$-component of \eqref{EQ:TN1} and using $\sigma,\gamma \in C_{\Gamma}(\Gamma_R)$, we obtain
\[
  c\,(x_{n-1})_{\sigma\gamma} - (x_{n-1})_{\sigma\gamma}\,c
  = n\,(x_n)_\sigma\,\delta(c).
\]
Since  $R$ is gr-simple, it is in particular $\delta$-gr-simple. Lemma \ref{LEM_QALGEBRA} therefore shows that $n\cdot 1_R$ is invertible. Since  $(x_n)_\sigma^{-1} \in Z(R)$, left-multiplying by   $(n\cdot1_R)^{-1}(x_n)_\sigma^{-1}$  yields
\[
  \delta(c) = r\,c - c\,r = \delta_r(c),
\]
where
  $
r := -(n\cdot1_R)^{-1}(x_n)_\sigma^{-1}(x_{n-1})_{\sigma\gamma}.
$
 Moreover, $r \in R_\gamma$ since $(x_n)_\sigma^{-1} \in R_{\sigma^{-1}}$
and $(x_{n-1})_{\sigma\gamma} \in R_{\sigma\gamma}$, hence
$r \in R_{\sigma^{-1}\sigma\gamma} = R_\gamma$. Thus $\delta = \delta_r$ with $r \in R_\gamma$, contradicting the assumption that $\delta$ is $\gamma$-outer.  Therefore $n = 0$, and $s = x_0 \in Z(R)$.
\end{proof}

\section{Graded ring-theoretic properties of
\texorpdfstring{$R[t;\delta]$}{R[t;delta]}}\label{SECORE111}
We begin the analysis of ideal-theoretic properties of $R[t;\delta]$ by reducing graded-ideal questions to suitable ideals of the coefficient ring via leading-coefficient ideals. We then treat gr-simplicity, gr-primality, and gr-Noetherianity in separate subsections.

\subsection{Gr-simplicity}\label{SECORE2435}

\begin{proposition}\label{POPTGRS123}
Let $R$ be a graded ring of characteristic zero and
$\delta \in \Der_\gamma(R)$.  If $R$ is gr-simple and $\delta$ is
$\gamma$-outer, then $R[t;\delta]$ is gr-simple.
\end{proposition}

\begin{proof}
Let $\mathfrak{a}$ be a nonzero graded ideal of $R[t;\delta]$, and set
$n := \min\{\deg_t(f) : f \in \mathfrak{a},\ f \neq 0\}$.  By Lemma \ref{LEMORE112}, $\mathcal{L}_n(\mathfrak{a})$ is a nonzero graded ideal of $R$, so gr-simplicity gives $\mathcal{L}_n(\mathfrak{a}) = R$.  In particular $1 \in R_e \subseteq
\mathcal{L}_n(\mathfrak{a})$, so there exists $g \in \mathfrak{a}$ of $t$-degree $n$ with leading coefficient $1$.  Write $g = \sum_{\tau \in \Gamma} g_\tau$ with $g_\tau \in R[t;\delta]_\tau$.
Since $\mathfrak{a}$ is a graded ideal, $g_\tau \in \mathfrak{a}_\tau$ for every $\tau$.  By \eqref{ORE130}, the $t^n$-coefficient of $g_\tau$
lies in $R_{\tau\gamma^{-n}}$; since the overall leading coefficient of $g$ is $1 \in R_e$, the only component $g_\tau$ with nonzero $t^n$-coefficient is $g_{\gamma^n}$, whose $t^n$-coefficient is $1 \in R_e$.  Set $f := g_{\gamma^n} \in \mathfrak{a}_{\gamma^n}$; then $f$ is monic of
$t$-degree $n$ with leading coefficient $1 \in R_e$.

For any $c \in \mathcal{H}(R)$, \eqref{ORET103} gives
\[
  t^n c = ct^n + n\,\delta(c)\,t^{n-1}
  + \sum_{k=2}^{n}\binom{n}{k}\delta^k(c)\,t^{n-k},
\]
so
\begin{align*}
  fc - cf
  &= \bigl(t^n + x_{n-1}t^{n-1} + \cdots + x_0\bigr)c
   - c\bigl(t^n + x_{n-1}t^{n-1} + \cdots + x_0\bigr) \\
  &= \bigl(n\,\delta(c) + x_{n-1}c - cx_{n-1}\bigr)t^{n-1}
   + \text{(terms of $t$-degree} < n-1\text{)}.
\end{align*}
Since $\mathfrak{a}$ is two-sided, $fc - cf \in \mathfrak{a}$, and
$\deg_t(fc - cf) \leq n-1 < n$, so minimality of $n$ forces $fc - cf = 0$.  Hence $f \in C_{R[t;\delta]}(R)$. The leading coefficient of $f$ is $1 \in R_e$.  Since $e \in Z(\Gamma)$ and $(1)_e = 1 \neq 0$, Lemma \ref{LEMORE110} applies and gives $n = 0$, hence $f = 1 \in \mathfrak{a}$ and $\mathfrak{a} = R[t;\delta]$.
\end{proof}

\begin{proposition}\label{THMORE113}
Let $R$ be a graded ring of characteristic zero and $\delta \in \Der_\gamma(R)$. If $R$ is $\delta$-gr-simple and $\delta$ is $\gamma$-outer, then $R[t;\delta]$ is gr-simple.
\end{proposition}

\begin{proof}
Let $\mathfrak{a}$ be a nonzero graded ideal of $R[t;\delta]$, and let $n$ be the minimal $t$-degree of a nonzero element of $\mathfrak{a}$.  By
Lemma \ref{LEMORE112}, $\mathcal{L}_n(\mathfrak{a})$ is a nonzero
$\delta$-invariant graded ideal of $R$.  Since $R$ is $\delta$-gr-simple, we have $\mathcal{L}_n(\mathfrak{a}) = R$, and hence $1 \in R_e \subseteq \mathcal{L}_n(\mathfrak{a})$.  By the same argument as in the proof of Proposition \ref{POPTGRS123}, there exists a monic element $f = t^n + x_{n-1}t^{n-1} + \cdots + x_0 \in \mathfrak{a}_{\gamma^n}$ of $t$-degree $n$ with leading coefficient $1 \in R_e$. As in that proof, we obtain
\[
  fc - cf = 0 \quad \text{for all } c \in R,
\]
and hence
\[ n\,\delta(c) + x_{n-1}c - cx_{n-1} = 0 \quad \text{for all } c \in R.
\]
If $n > 0$, then   Lemma \ref{LEM_QALGEBRA} shows that $n\cdot1_R$ is invertible. Since
 $x_{n-1} \in R_{\gamma^n \gamma^{-(n-1)}} = R_\gamma$, we   obtain
 \[
  \delta(c) = cr - rc,
\]
with    $
r = -(n\cdot1_R)^{-1}x_{n-1} \in R_\gamma,
$  contradicting the assumption that $\delta$ is $\gamma$-outer.  Therefore $n = 0$, hence $f = 1 \in \mathfrak{a}$, and $\mathfrak{a} = R[t;\delta]$.
\end{proof}

\begin{proposition}\label{POPOORE126}
Let $R$ be a graded ring and $\delta \in \Der_\gamma(R)$. If $R[t;\delta]$ is gr-simple, then $R$ is $\delta$-gr-simple and $\delta$ is $\gamma$-outer.
\end{proposition}

\begin{proof}
Let $\mathfrak{a} \subsetneq R$ be a nonzero $\delta$-invariant graded ideal. The set \[ \mathfrak{a}[t;\delta] := \Bigl\{\sum_{i \geq 0} x_i t^i : x_i \in \mathfrak{a}\Bigr\}\subseteq R[t;\delta]
\] is closed under addition, under left and right multiplication by elements of $R$ (since $\mathfrak{a}$ is a two-sided ideal of $R$), and under left multiplication by $t$ since $tx = xt + \delta(x)$ with $x, \delta(x) \in \mathfrak{a}$ by $\delta$-invariance. It is graded: for any $f = \sum_{i \geq 0} x_i t^i \in \mathfrak{a}[t;\delta]$, the $\tau$-homogeneous component of $f$ in $R[t;\delta]$ is $f_\tau = \sum_{i \geq 0}(x_i)_{\tau\gamma^{-i}}\,t^i$ by \eqref{ORE130}; since $\mathfrak{a}$ is a graded ideal of $R$, each $(x_i)_{\tau\gamma^{-i}} \in \mathfrak{a}$, so $f_\tau \in \mathfrak{a}[t;\delta]$. Hence $\mathfrak{a}[t;\delta]$ is a nonzero graded ideal of $R[t;\delta]$.  Since $1 \notin \mathfrak{a}$, it is proper, contradicting gr-simplicity of $R[t;\delta]$.

Suppose now that $\delta = \delta_r$ for some $r \in R_\gamma$.  Since $rx - xr = \delta(x)$ for all $x \in R$, Theorem \ref{THMORE105}
yields a graded ring homomorphism $\mathrm{ev}_r \colon R[t;\delta] \to R$ with $\mathrm{ev}_{r|_R} = \mathrm{id}$ and $\mathrm{ev}_r(t) = r$.  Then $t - r \in \ker(\mathrm{ev}_r)$ and $t - r \neq 0$, so $\ker(\mathrm{ev}_r)$ is a proper nonzero graded ideal of $R[t;\delta]$, contradicting gr-simplicity.
\end{proof}

Theorem \ref{THMORE02}(2) follows from Propositions \ref{THMORE113} and \ref{POPOORE126}.

\begin{corollary}\label{CORORE114}
If $\charr R= 0$, then $R[t;\delta]$ is gr-simple if and only if $R$ is $\delta$-gr-simple and $\delta$ is $\gamma$-outer.
\end{corollary}

The following example shows that the condition
$\gamma\in C_\Gamma(\Gamma_R)$ is genuinely weaker than requiring $\gamma\in Z(\Gamma)$, and that the gr-simplicity criterion applies in this situation.

\begin{example}
Let $\mathbb F$ be a field of characteristic zero and let
$\Gamma=D_\infty=\langle a,s\mid s^2=e,\;sas^{-1}=a^{-1}\rangle$.
Set
$R=M_2(\mathbb F)[x,x^{-1}]$, where $x$ is central, and grade $R$ by
$R_{a^k}=M_2(\mathbb F)x^k$ for $k\in\mathbb Z$, with
$R_\tau=0$ for $\tau\notin\langle a\rangle$.
Then $\Gamma_R=\langle a\rangle$ and
\[
C_\Gamma(\Gamma_R)=\langle a\rangle,
\quad Z(\Gamma)=\{e\}. 
\]
In particular, $\gamma:=a^{-1}$ belongs to
$C_\Gamma(\Gamma_R)\setminus Z(\Gamma)$.

Define $\delta\colon R\to R$ by
$\delta(Mx^k)=kMx^{k-1}$ for $M\in M_2(\mathbb F)$ and $k\in\mathbb Z$. This is a derivation, and
$\delta(R_{a^k})\subseteq R_{a^{k-1}}=R_{a^k\gamma}$, so $\delta\in\Der_\gamma(R)$.

We claim that $R$ is gr-simple. Indeed, let $\mathfrak a$ be a nonzero graded ideal and choose a nonzero homogeneous element $Mx^k\in\mathfrak a$. Since $M_2(\mathbb F)$ is simple, there exist
$A_i,B_i\in M_2(\mathbb F)$ such that $\sum_i A_iMB_i=I_2$. Hence \[
I_2=\sum_i A_i(Mx^k)(B_i x^{-k})\in\mathfrak a,
\]
and therefore $\mathfrak a=R$. In particular, $R$ is
$\delta$-gr-simple.

Finally, $\delta$ is $\gamma$-outer. Indeed, $xI_2\in Z(R)$, whereas $\delta(xI_2)=I_2\neq0$. Since every inner derivation vanishes on $Z(R)$, the derivation $\delta$ cannot be inner. Thus, by Theorem \ref{THMORE02}(2), $R[t;\delta]$ is gr-simple.
\end{example}

The following   two propositions give the center criterion appearing in  Theorem \ref{THMORE02}(3).

\begin{proposition}\label{PORE115}
Let $R$ be a $\delta$-gr-simple ring and
$\delta \in \Der_\gamma(R)$.  Every nonzero graded ideal $\mathfrak{a}$ of $R[t;\delta]$ satisfies $\mathfrak{a}\cap Z(R[t;\delta])\neq 0$.
\end{proposition}

\begin{proof}
Set $n := \min\{\deg_t(f) : f \in \mathfrak{a},\ f \neq 0\}$.  By Lemma \ref{LEMORE112}, $\mathcal{L}_n(\mathfrak{a})$ is a nonzero
$\delta$-invariant graded ideal of $R$, so $\delta$-gr-simplicity gives $\mathcal{L}_n(\mathfrak{a}) = R$.  By the argument of Proposition~\ref{THMORE113}, there exists a monic element
\[f = t^n + x_{n-1}t^{n-1} + \cdots + x_0 \in \mathfrak{a}_{\gamma^n}
\] with $f \in C_{R[t;\delta]}(R)$.  It remains to show $ft = tf$.  We have \[ft = t^{n+1} + \sum_{i=0}^{n-1} x_i\,t^{i+1}, \quad
tf = t^{n+1} + \sum_{i=0}^{n-1}\bigl(x_i\,t^{i+1} + \delta(x_i)\,t^i\bigr).\] So $ft - tf = -\sum_{i=0}^{n-1}\delta(x_i)\,t^i \in \mathfrak{a}$.  Since $\deg_t(ft - tf) \leq n-1 < n$, minimality gives $ft = tf$.  Since $R$ and $t$ generate $R[t;\delta]$, it follows that $f \in Z(R[t;\delta])$.
\end{proof}

\begin{proposition}\label{THMORE117}
Let $R$ be a graded ring and $\delta \in \Der_\gamma(R)$.   Assume that $Z(R[t;\delta])$ is a graded subring of $R[t;\delta]$. Then the  following are equivalent.
\begin{itemize}
\item[(1)] $R[t;\delta]$ is gr-simple.
\item[(2)] $R$ is $\delta$-gr-simple and $Z(R[t;\delta])$ is a graded field.
\end{itemize}
\end{proposition}

\begin{proof}
$(1) \Rightarrow (2)$. Proposition \ref{POPOORE126} gives $\delta$-gr-simplicity of $R$.  For any nonzero homogeneous $f \in Z(R[t;\delta])$, the ideal $fR[t;\delta]$ is a nonzero graded ideal of $R[t;\delta]$, so
$fR[t;\delta] = R[t;\delta]$ and $f \in R[t;\delta]^\times$.  Hence
$Z(R[t;\delta])$ is a graded field.

$(2) \Rightarrow (1)$.  Let $\mathfrak{a}$ be a nonzero graded ideal of $R[t;\delta]$.  By Proposition \ref{PORE115}, $\mathfrak{a} \cap Z(R[t;\delta]) \neq 0$; pick a nonzero homogeneous $f \in \mathfrak{a} \cap Z(R[t;\delta])$.  Since $Z(R[t;\delta])$ is a graded
field, $f$ is invertible, so $1 \in \mathfrak{a}$ and
$\mathfrak{a} = R[t;\delta]$.
\end{proof}

\begin{remark}\label{REMCENTERGRADED}
The hypothesis in Proposition \ref{THMORE117} concerns the grading on the center, not the ordinary field property of the center. For a general group grading, $Z(R[t;\delta])$ need not be a graded subring, so the phrase ``$Z(R[t;\delta])$ is a graded field'' requires this point to be addressed. The hypothesis is automatic when $\Gamma$ is abelian \cite[Lemma 2.2]{Oinert2026central}. More generally, \cite[Proposition 2.1]{Ion1988} applies when $\Gamma$ admits a total order invariant under left and right multiplication and every nonzero homogeneous element of $R[t;\delta]$ is regular. Under either of these conditions, Proposition \ref{THMORE117} gives the characteristic-free graded analogue of the Oinert--Silvestrov criterion.
\end{remark}

The following corollary describes how gr-simplicity behaves under the isomorphisms of Corollaries \ref{CORORE107}--\ref{COROORE109}, proving
Theorem \ref{THMORE02}(4).

\begin{corollary}\label{CORORE118}
\begin{itemize}
\item[(1)] If $\delta_1$ and $\delta_2$ are conjugate $\gamma$-derivations, then $R[t;\delta_1]$ is gr-simple if and only if $R[t;\delta_2]$ is.
\item[(2)]  If $\delta = \delta_r + \delta_2$ with $r \in R_\gamma$ and
$\delta_2 \in \Der_\gamma(R)$, then $R[t;\delta]$ is gr-simple if and only if $R[t;\delta_2]$ is.
\item[(3)] Let $T$ be a graded ring, $\delta \in \Der_\gamma(T)$, and
$S := M_n(T)$ with the induced grading and entrywise $\gamma$-derivation
$\hat\delta$. Then $S[t;\hat\delta]$ is gr-simple if and only if $T[t;\delta]$ is.
\end{itemize}
\end{corollary}

\begin{proof}
Parts (1) and (2) follow from Corollaries \ref{CORORE107} and \ref{CORORE1078}, since gr-simplicity is preserved under graded isomorphisms.

For (3), Corollary \ref{COROORE109} gives
$M_n(T)[t;\hat\delta] \cong_{\mathrm{gr}} M_n(T[t;\delta])$.  For any graded ring $B$, the map $\mathfrak{b} \mapsto M_n(\mathfrak{b})$ is a lattice isomorphism between the graded ideals of $B$ and those of $M_n(B)$.  Hence $M_n(B)$ is gr-simple if and only if $B$ is.  Applying this with $B = T[t;\delta]$ gives the result.
\end{proof}

One may ask whether the construction $R\mapsto R[t;\delta]$ preserves
gr-simplicity under iteration.  The first step already fails: $R[t]$ is never gr-simple, since $tR[t]$ is a proper graded ideal. The derivation $\delta := \tfrac{d}{dt}$ repairs this at the second step.  This is Theorem \ref{THMORE02}(5).

\begin{proposition}\label{THMORE128}
Let $R$ be a gr-simple graded ring of characteristic zero, and let
$R' := R[t; 0]$ with $R'_\tau := R_\tau[t]$ and $\deg(t) = e$.  Then
$R'\bigl[s;\tfrac{d}{dt}\bigr]$ is gr-simple.
\end{proposition}

\begin{proof}
Let $\delta := \tfrac{d}{dt} \colon R' \to R'$ denote the standard derivation. The defining relation in $R'\bigl[s;\,\delta\bigr]$ is $sx = xs + \delta(x)$ for all $x \in R'$. In particular, $st = ts + 1$ and $sx = xs$ for all $x \in R$.

Let $\mathfrak{a}$ be a nonzero graded ideal of $R'\bigl[s;\,\delta\bigr]$ and let $f \in \mathfrak{a}$ be a nonzero element. Write
\[
  f = \sum_{i=0}^{n} t^i\, p_i(s),
\]
where $p_i(s) = \sum_{j} x_{ij}\, s^j \in R[s]$ and $p_n(s) \neq 0$.
For any $x \in R$ and $k \geq 0$, the defining relation gives
$[s,\, t^k x] = \delta(t^k x) = k\, t^{k-1} x$, and $[s, p_i(s)] = 0$
since $\delta$ vanishes on $R$ and $s$ commutes with itself. Therefore,
setting $\mathrm{ad}_s(f) := [s,f]$,
\[
  \mathrm{ad}_s^n(t^i\, p_i(s)) =
  \begin{cases}
    n!\, p_n(s) & \text{if } i = n, \\
    0           & \text{if } i < n.
  \end{cases}
\]
Since $\mathfrak{a}$ is a two-sided ideal and $s \in R'\bigl[s;\,\delta\bigr]$, we have $\mathrm{ad}_s^n(f) = n!\, p_n(s) \in \mathfrak{a}$. As $\operatorname{char}(R) = 0$, the element $n! \cdot 1_R$ is invertible in $R$, so $p_n(s) \in \mathfrak{a}$.

Write $p_n(s) = \sum_{j=0}^{m} x_j\, s^j$ with $x_j \in R$ and $x_m \neq 0$. From $st = ts + 1$ we get $[t,s] = -1$, and since $[t,x] = 0$ for all $x \in R$, the Leibniz rule gives $[t, s^j] = -j\, s^{j-1}$, hence
$\mathrm{ad}_t(x_j s^j) = -j\, x_j s^{j-1}$. Iterating $m$ times,
\[
  \mathrm{ad}_t^m\bigl(p_n(s)\bigr) = (-1)^m\, m!\, x_m.
\]
Since $p_n(s) \in \mathfrak{a}$ and $t \in R'\bigl[s;\,\delta\bigr]$,
we have $\mathrm{ad}_t^m(p_n(s)) \in \mathfrak{a}$, and since $m!$ is invertible in $R$, it follows that $x_m \in \mathfrak{a}$. The element $x_m$ is nonzero and belongs to $R_\sigma$ for some $\sigma \in \Gamma_R$, so $Rx_mR$ is a nonzero graded ideal of $R$. Since $R$ is gr-simple, $Rx_mR = R$, hence $1_R \in \mathfrak{a}$, and therefore $\mathfrak{a} = R'\bigl[s;\,\tfrac{d}{dt}\bigr]$.
\end{proof}

\begin{remark}
Alternatively, one can verify directly that $\delta = \tfrac{d}{dt}$ is $e$-outer on $R'$ and that $R'$ is $\delta$-gr-simple, and then conclude by Corollary \ref{CORORE114}.
\end{remark}

\subsection{Gr-primality}\label{SECORE125}

\begin{proposition}\label{THMORE124}
Let $R$ be a $\delta$-gr-prime graded ring and
$\delta \in \Der_\gamma(R)$.  Then $R[t;\delta]$ is gr-prime.
\end{proposition}

\begin{proof}
Suppose for a contradiction that $\mathfrak{a}$ and $\mathfrak{b}$ are nonzero graded ideals of $R[t;\delta]$ with $\mathfrak{a}\mathfrak{b} = 0$.  Let $n := \min\bigl\{\deg_t(f) : f \in \mathfrak{a},\, f \neq 0\bigr\}$ and  $m := \min\bigl\{\deg_t(g) : g \in \mathfrak{b},\, g \neq 0\bigr\}$. By Lemma \ref{LEMORE112}, $\mathcal{L}_n(\mathfrak{a})$ and
$\mathcal{L}_m(\mathfrak{b})$ are nonzero $\delta$-invariant graded ideals of $R$. Let $x\in\mathcal{L}_n(\mathfrak{a})_\sigma$ and
$x'\in\mathcal{L}_m(\mathfrak{b})_\tau$ be homogeneous, and pick
\[ f = xt^n + x_{n-1}t^{n-1} + \cdots + x_0 \in \mathfrak{a}_{\sigma\gamma^n},
  \quad
  g = x't^m + x'_{m-1}t^{m-1} + \cdots + x'_0 \in \mathfrak{b}_{\tau\gamma^m}.
\]
Since $\mathfrak{a}\mathfrak{b} = 0$, we have $fg = 0$. By \eqref{ORET103},
\[
  (xt^n)(x't^m) = xx'\,t^{n+m} + \sum_{k=1}^{n}\binom{n}{k}x\,\delta^k(x')\,t^{n+m-k}.
\]
Every other summand $(x_i t^i)(x'_j t^j)$ with $i < n$ or $j < m$ contributes only to $t$-degrees strictly less than $n + m$, so the $t^{n+m}$-coefficient of $fg$ is $xx'$.  Since $fg = 0$, we get $xx' = 0$.  As $x\in\mathcal{L}_n(\mathfrak{a})$ and $x'\in\mathcal{L}_m(\mathfrak{b})$ were arbitrary homogeneous elements, $\mathcal{L}_n(\mathfrak{a})\,\mathcal{L}_m(\mathfrak{b}) = 0$, contradicting $\delta$-gr-primality of $R$.
\end{proof}

The above proposition proves Theorem \ref{THM35}(1) in the stronger form that $\delta$-gr-primality of $R$ suffices; the original statement of Theorem \ref{THM35}(1) with gr-prime follows as a special case.

\begin{proposition}\label{POPP123ORE}
If $R[t;\delta]$ is gr-prime, then $R$ is $\delta$-gr-prime.
\end{proposition}

\begin{proof}
Let $\mathfrak{a}$ and $\mathfrak{b}$ be nonzero $\delta$-invariant graded ideals of $R$ with $\mathfrak{a}\mathfrak{b} = 0$.  As in Proposition \ref{POPOORE126}, the extensions $\mathfrak{a}[t;\delta]$ and $\mathfrak{b}[t;\delta]$ are nonzero graded ideals of $R[t;\delta]$. For any $f = \sum_i x_i t^i \in \mathfrak{a}[t;\delta]$ and
$g = \sum_j y_j t^j \in \mathfrak{b}[t;\delta]$, \eqref{ORET103} gives
\[(x_i t^i)(y_j t^j) = \sum_{k=0}^{i}\binom{i}{k}x_i\,\delta^k(y_j)\,t^{i+j-k}.
\]
Since $\mathfrak{b}$ is $\delta$-invariant, $\delta^k(y_j) \in \mathfrak{b}$ for every $k \geq 0$, so $x_i\,\delta^k(y_j) \in \mathfrak{a}\mathfrak{b} = 0$. Hence $fg = 0$ and $\mathfrak{a}[t;\delta]\,\mathfrak{b}[t;\delta] = 0$. Since $R[t;\delta]$ is gr-prime, $\mathfrak{a}[t;\delta] = 0$ or $\mathfrak{b}[t;\delta] = 0$, giving $\mathfrak{a} = 0$ or $\mathfrak{b} = 0$.
\end{proof}

Combining Propositions \ref{THMORE124} and \ref{POPP123ORE} gives Theorem \ref{THM35}(2).

\begin{corollary}
$R[t;\delta]$ is gr-prime if and only if $R$ is $\delta$-gr-prime.
\end{corollary}

\subsection{Gr-Noetherianity}  
We prove that gr-Noetherianity of $R$ passes to $R[t;\delta]$
(Theorem \ref{THMORE03}). The argument is a graded version of the standard Hilbert basis theorem
proof: we introduce right leading-coefficient ideals $\mathcal{L}_n^r(\mathfrak{a})$ and show that any ascending chain of graded right ideals of $R[t;\delta]$ stabilizes by comparing the stabilization indices of the finitely many chains $\bigl(\mathcal{L}_k^r(\mathfrak{a}_j)\bigr)_j$ for $k < s_0$ together with the eventual constancy established in \eqref{ORE1234}.
Taking $\delta = 0$ recovers the graded Hilbert basis theorem as Corollary \ref{cor211A}.

For a graded right ideal $\mathfrak{a} \subseteq R[t;\delta]$ and $n \geq 0$, define \[
  \mathcal{L}_n^{r}(\mathfrak{a})
  := \{0\}
  \cup \bigl\{x \in R : x
    \text{ is the leading coefficient  of some } f\in \mathfrak{a}
  \bigr\}.
\]

\begin{lemma}\label{LEMORE120}
Let $\mathfrak{a} \subseteq \mathfrak{b}$ be graded right ideals of $R[t;\delta]$ and $n \geq 0$.
\begin{itemize}
\item[(1)] $\mathcal{L}_n^{r}(\mathfrak{a})$ is a graded right ideal of $R$.
\item[(2)] $\mathcal{L}_n^{r}(\mathfrak{a}) \subseteq \mathcal{L}_n^{r}(\mathfrak{b})$.
\item[(3)] If $\mathcal{L}_n^{r}(\mathfrak{a}) = \mathcal{L}_n^{r}(\mathfrak{b})$ for all  $n \geq 0$, then $\mathfrak{a} = \mathfrak{b}$.
\end{itemize}
\end{lemma}

\begin{proof}
Part (1) follows from Lemma \ref{LEMORE112}. For part (2),
let $x \in \mathcal{L}_n^r(\mathfrak{a})$. Then $x$ is the leading coefficient of some $f \in \mathfrak{a}$ of $t$-degree $n$. Since $\mathfrak{a} \subseteq \mathfrak{b}$, we have $f \in \mathfrak{b}$,
whence $x \in \mathcal{L}_n^r(\mathfrak{b})$.
For part (3), since $\mathfrak{a} \subseteq \mathfrak{b}$, it suffices to show $\mathfrak{b} \subseteq \mathfrak{a}$. We proceed by induction on
$\deg_t(g)$ for $g \in \mathfrak{b}$. If $\deg_t(g) = 0$, then $g \in R$ and $g$ is its own leading coefficient, so $g \in \mathcal{L}_0^r(\mathfrak{b}) = \mathcal{L}_0^r(\mathfrak{a})$. Hence there exists $f \in \mathfrak{a}$ of $t$-degree $0$ with leading coefficient $g$; since $f$ has $t$-degree $0$, we have $f = g$, so $g \in \mathfrak{a}$. Suppose the claim holds for all elements of $\mathfrak{b}$ of $t$-degree less than $n$, and let $g \in \mathfrak{b}$ with $\deg_t(g) = n$ and leading coefficient $x$. Then
$x \in \mathcal{L}_n^r(\mathfrak{b}) = \mathcal{L}_n^r(\mathfrak{a})$, so there exists $f \in \mathfrak{a}$ of $t$-degree $n$ with leading coefficient $x$. The difference $g - f$ lies in $\mathfrak{b}$ and
satisfies $\deg_t(g - f) < n$, so the induction hypothesis gives $g - f \in \mathfrak{a}$. Since $f \in \mathfrak{a}$, we conclude
$g \in \mathfrak{a}$.
\end{proof}

\begin{proposition}\label{THMORE121}
Let $\delta \in \Der_\gamma(R)$.
\begin{itemize}
\item[(1)]  If $R$ is gr-right-Noetherian, then so is $R[t;\delta]$.
\item[(2)]  If $R$ is gr-left-Noetherian, then so is $R[t;\delta]$.
\end{itemize}
\end{proposition}

\begin{proof}
We prove (1); part (2) follows by replacing graded right ideals with graded left ideals and $\mathcal{L}_n^r$ with its left analogue $\mathcal{L}_n^\ell$ throughout.

Let
\[
  \mathfrak{a}_1 \subseteq \mathfrak{a}_2 \subseteq \mathfrak{a}_3 \subseteq \cdots
\] be an ascending chain of graded right ideals of $R[t;\delta]$.  We show this chain stabilizes.

We construct an ascending chain of graded right ideals of $R$ by coupling the index of the ideal with the degree of the leading coefficient. For any homogeneous $f = at^s + \cdots \in \mathfrak{a}_s$, the product $ft \in \mathfrak{a}_s$ has $t$-degree $s+1$ and leading coefficient $a$, so $\mathcal{L}_s^r(\mathfrak{a}_s) \subseteq \mathcal{L}_{s+1}^r(\mathfrak{a}_s) \subseteq \mathcal{L}_{s+1}^r(\mathfrak{a}_{s+1})$,
the second inclusion by Lemma~\ref{LEMORE120}(2). Hence
\[
  \mathcal{L}_1^r(\mathfrak{a}_1) \subseteq \mathcal{L}_2^r(\mathfrak{a}_2) \subseteq \cdots
\]
is an ascending chain of graded right ideals of $R$. Since $R$ is gr-right-Noetherian, there exists $s_0 \geq 1$ such that
\begin{equation}\label{ORE1234}
  \mathcal{L}_{s_0}^r(\mathfrak{a}_{s_0})
  = \mathcal{L}_{s_0+n}^r(\mathfrak{a}_{s_0+n})
  \quad \text{for all } n \geq 0.
\end{equation}
We claim that $\mathcal{L}_k^r(\mathfrak{a}_j) = \mathcal{L}_{s_0}^r(\mathfrak{a}_{s_0})$ for all $k \geq s_0$ and $j \geq s_0$. For $a \in \mathcal{L}_{s_0}^r(\mathfrak{a}_{s_0})$,
pick $f \in \mathfrak{a}_{s_0} \subseteq \mathfrak{a}_j$ of degree $s_0$ with leading coefficient $a$, then $ft^{k-s_0} \in \mathfrak{a}_j$ has degree $k$ with leading coefficient $a$, so $a \in \mathcal{L}_k^r(\mathfrak{a}_j)$.

For the converse inclusion, let $m := \max(j,k)$. Then
\[
  \mathcal{L}_k^r(\mathfrak{a}_j)
  \subseteq \mathcal{L}_k^r(\mathfrak{a}_m)
  \subseteq \mathcal{L}_m^r(\mathfrak{a}_m)
  = \mathcal{L}_{s_0}^r(\mathfrak{a}_{s_0}),
\]
where the first inclusion is Lemma~\ref{LEMORE120}(2), the second uses the argument above applied to $\mathfrak{a}_m$, and the equality is \eqref{ORE1234}.

For each $k \in \{0, \ldots, s_0-1\}$, the chain
\[\mathcal{L}_k^r(\mathfrak{a}_1) \subseteq \mathcal{L}_k^r(\mathfrak{a}_2) \subseteq \cdots\]
is an ascending chain of graded right ideals of $R$, so it stabilizes at some $\nu_k \geq 1$. Let $j_0 := \max(\nu_0, \ldots, \nu_{s_0-1}, s_0)$. Then $\mathcal{L}_k^r(\mathfrak{a}_j) = \mathcal{L}_k^r(\mathfrak{a}_{j_0})$ for all $k \geq 0$
and $j \geq j_0$: for $k < s_0$ by stabilization, and for $k \geq s_0$ by the claim. Therefore, by Lemma \ref{LEMORE120}(3), $\mathfrak{a}_j = \mathfrak{a}_{j_0}$ for all $j \geq j_0$.
\end{proof}

\begin{corollary}\label{cor211A}
If $R$ is gr-left-Noetherian (resp.\ gr-right-Noetherian), then so is $R[t]$.
\end{corollary}

\begin{proof}
Apply Proposition \ref{THMORE121} with $\delta = 0$.
\end{proof}

\section{Ring-theoretic invariance under graded equivalence}\label{SECMORITA}
Throughout this section, $\Gamma$ denotes an \emph{abelian} group.
We establish that gr-simplicity and gr-primality are invariants of
the homogeneous graded Morita equivalence relation $\approx_{\mathrm{hgr}}$ of Definition \ref{GORE100}, and that the differential polynomial structure of $R[t;\delta]$ is itself preserved under equivalence via a compatible idempotent. The abelian hypothesis is imposed so that we may apply the Homogeneously Graded Extended Morita Theorem of Abrams, Ruiz, and Tomforde \cite{abrams2022morita}; as noted in the introduction, the algebraic core of that framework (the equivalence (C2)$\Leftrightarrow$(C4)) does not require commutativity of $\Gamma$, and the full extension to arbitrary groups is expected once a proof of \cite[Theorem 4.7]{abrams2022morita}
for non-abelian groups is available.

\begin{proposition}\label{PORPOREMORITA123}
Let $A$ and $B$ be $\Gamma$-graded rings with $A \approx_{\mathrm{hgr}} B$. Then $A$ is gr-simple \textup{(}resp.\ gr-prime\textup{)} if and only if $B$ is.
\end{proposition}
\begin{proof}
By Definition \ref{GORE100}, there exist $n \in \NN$, a full homogeneous
idempotent $\ell \in M_n(B)_e$, and a graded isomorphism
$A \cong_{\mathrm{gr}} \ell M_n(B)\ell$.  Set $R := M_n(B)$.

Suppose $A$ is gr-simple, and let $\mathfrak{b}$ be a nonzero proper graded ideal of $B$.  Set $\mathfrak{d} := M_n(\mathfrak{b})$.  Then $\mathfrak{d}$ is a nonzero proper graded ideal of $R$, and $\mathfrak{a} := \ell\mathfrak{d}\ell$
is a graded ideal of $\ell R\ell$.  Since $\ell$ is full, $\mathfrak{a}$ is a nonzero graded ideal of $\ell R\ell \cong A$.  By gr-simplicity of $A$, we have $\mathfrak{a}=\ell R\ell$.  Then $\ell = \ell 1_R\ell \in \ell R\ell = \ell\mathfrak{d}\ell$.  Since $\ell$ is full,
$R = R\ell R \subseteq R(\ell\mathfrak{d}\ell)R \subseteq \mathfrak{d}$, so $R=\mathfrak{d}$, a contradiction.  Therefore $B$ is gr-simple.  The converse follows by symmetry.

Suppose that $A$ is gr-prime, and let $\mathfrak{b}_1, \mathfrak{b}_2$ be
nonzero graded ideals of $B$ such that $\mathfrak{b}_1\mathfrak{b}_2 = 0$. For $i = 1, 2$, set $\mathfrak{a}_i := \ell M_n(\mathfrak{b}_i)\ell$.
As above, each $\mathfrak{a}_i$ is a nonzero graded ideal of $A$.  Moreover,
\[
\mathfrak{a}_1\mathfrak{a}_2
= \ell M_n(\mathfrak{b}_1)\ell \, \ell M_n(\mathfrak{b}_2)\ell
\subseteq \ell M_n(\mathfrak{b}_1\mathfrak{b}_2)\ell = 0,
\]
which contradicts the gr-primality of $A$.  Hence $\mathfrak{b}_1 = 0$ or
$\mathfrak{b}_2 = 0$, and $B$ is gr-prime.  The converse again follows by symmetry.
\end{proof}

The following proposition characterizes the graded rings homogeneously graded equivalent to $R[t;\delta]$ that arise from full idempotents in $R_e$, showing that they are again graded differential polynomial rings.

\begin{proposition}\label{POPDP2563t}
Let $T$ be a graded ring,  $\delta$ a $\gamma$-derivation of $T$, and $\ell \in T_e$ an idempotent that is full in $T_e$ and satisfies $\delta(\ell) = 0$.  Define $S := \ell T\ell$ with induced grading $S_\tau := \ell T_\tau\ell$ for $\tau \in \Gamma_R$, and
the map $\bar{\delta} \colon S \to S$ by $\bar{\delta}(s) := \ell\delta(s)\ell$ for $s\in S$.  Then the following statements hold.
\begin{itemize}
\item[(1)] $\bar{\delta} \in \Der_\gamma(S)$.
\item[(2)] $\ell \, T[t; \delta] \, \ell \cong_{\mathrm{gr}} S[t; \bar{\delta}]$
as graded rings.
\item[(3)] $S[t; \bar{\delta}] \approx_{\mathrm{hgr}} T[t; \delta]$.
\end{itemize}
\end{proposition}

\begin{proof}
$(1)$ Let $s = \ell x \ell \in S_\tau$ for some $x \in T_\tau$.  Since $\delta(\ell) = 0$ and $\ell^2 = \ell$, the Leibniz rule yields $\delta(\ell^2) = 0$.  Thus,
\[
\bar{\delta}(s) = \ell \delta(\ell x \ell) \ell =
\ell \bigl( \delta(\ell) x\ell + \ell \delta(x) \ell + \ell x\delta(\ell) \bigr) \ell
= \ell \delta(x) \ell.
\]
As $\delta \in \Der_\gamma(T)$, we have $\delta(x) \in T_{\tau\gamma}$, so $\ell\delta(x)\ell \in S_{\tau\gamma}$.  Moreover, $\bar{\delta}$ satisfies the Leibniz rule
\[
\bar{\delta}(s_1 s_2) = \ell\delta(s_1 s_2)\ell =
\ell\bigl(\delta(s_1)s_2 + s_1\delta(s_2)\bigr)\ell =
\bigl(\ell\delta(s_1)\ell\bigr)\bigl(\ell s_2\ell\bigr) +
\bigl(\ell s_1\ell\bigr)\bigl(\ell\delta(s_2)\ell\bigr) =
\bar{\delta}(s_1)s_2 + s_1\bar{\delta}(s_2),
\]
since $s_i = \ell s_i = s_i\ell$ for $s_i \in S$.  Hence
$\bar{\delta} \in \Der_\gamma(S)$.

$(2)$ The Ore extension $T[t;\delta]$ is defined by the commutation relation $tx = xt + \delta(x)$ for $x\in T$, with grading
$(T[t;\delta])_\tau = \bigoplus_{i\geq 0} T_{\tau\gamma^{-i}}\,t^i$.
Since $\delta(\ell) = 0$, we have $t\ell = \ell t$.  By induction on $i$,
using $\delta(\ell) = 0$ with \eqref{ORET103}, we then get
\begin{equation}\label{Ore1235TY}
    t^i \ell =  \ell t^i, \quad i \geq 0.
\end{equation}
Define $\Phi \colon S[t; \bar{\delta}] \to \ell T[t; \delta] \ell$ by
\[
\Phi\left( \sum_{i=0}^n s_i t^i \right) := \sum_{i=0}^n s_i t^i,
\]
where each $s_i \in S \subseteq T$ (viewed in $T[t; \delta]$).

For $s_i = \ell x_i \ell \in S$, we have $\ell s_i = s_i = s_i \ell$.  Thus,
\[
\ell \left( \sum_i s_i t^i \right) \ell = \sum_i (\ell s_i) (t^i \ell) =
\sum_i s_i (\ell t^i) = \sum_i s_i t^i = \Phi\left( \sum_i s_i t^i \right),
\]
by \eqref{Ore1235TY}.  Hence $\Phi(S[t; \bar{\delta}]) \subseteq \ell T[t; \delta] \ell$.
The map $\Phi$ is clearly additive.  For multiplication in $S[t; \bar{\delta}]$, $ts = st + \bar{\delta}(s)$.
Then
\[
\Phi(ts) = \Phi(st + \bar{\delta}(s)) = st + \bar{\delta}(s) = st + \ell\delta(s)\ell.
\]
In $T[t;\delta]$, $ts = st + \delta(s)$, so
$\Phi(ts) = ts - \delta(s) + \ell\delta(s)\ell = ts + (\ell\delta(s)\ell - \delta(s))$.
But since $s = \ell x\ell$ for some $x\in T$ and $\delta(\ell)=0$,
\[
\delta(s) = \delta(\ell x\ell) = \ell\delta(x)\ell = \bar{\delta}(s),
\]
so $\Phi(ts) = ts = t\Phi(s)$.  Similarly, $\Phi(s_1 s_2) = \Phi(s_1)\Phi(s_2)$.
Thus $\Phi$ is a graded ring homomorphism.

For surjectivity, let $f = \sum_i x_i t^i \in T[t;\delta]$ with $\ell f\ell\in\ell T[t;\delta]\ell$.
Then
\[
\ell f\ell = \sum_i (\ell x_i\ell) t^i = \sum_i (\ell x_i\ell)(\ell t^i) =
\Phi\left(\sum_i (\ell x_i\ell) t^i\right),
\]
by \eqref{Ore1235TY}.  So $\Phi$ is surjective.  Injectivity follows directly.

$(3)$ We show that $\ell$ is full in $(T[t;\delta])_e$.  Since $\ell$ is full in $T_e$, there exist $a_j, b_j \in T_e$ with $\sum_j a_j \ell b_j = 1$. Let $r t^k$ be an arbitrary pure element of $(T[t;\delta])_e$, with
$r \in T_{\gamma^{-k}}$.  Then
\[rt^k
  = rt^k  \Bigl(\sum_j a_j\,\ell\, b_j\Bigr)
  = \sum_j \bigl(r  t^k  a_j\bigr) \ell  b_j.
\]
By \eqref{ORET103}, $t^k  a_j = \sum_{m=0}^{k}\binom{k}{m}\delta^m(a_j)\,t^{k-m}$ with $\delta^m(a_j) \in T_{\gamma^m}$, so each summand $r \delta^m(a_j) t^{k-m}$ lies in $T_{\gamma^{m-k}} t^{k-m}
\subseteq (T[t;\delta])_e$; hence $r  t^k  a_j \in (T[t;\delta])_e$.
Since $b_j \in T_e \subseteq (T[t;\delta])_e$, it follows that
$(T[t;\delta])_e  \ell  (T[t;\delta])_e$ contains every pure element
of $(T[t;\delta])_e$, so $\ell$ is full in $(T[t;\delta])_e$.
By \cite[Theorem 2.14]{abrams2022morita},
$M_\infty(T[t;\delta]) \cong_{\mathrm{gr}} M_\infty(\ell T[t;\delta]\ell)$. Part (2) gives $\ell T[t;\delta]\ell \cong_{\mathrm{gr}} S[t;\bar{\delta}]$, so
$M_\infty(S[t;\bar{\delta}]) \cong_{\mathrm{gr}} M_\infty(T[t;\delta])$.
Therefore $S[t;\bar{\delta}]\approx_{\mathrm{hgr}} T[t;\delta]$.
\end{proof}

\begin{remark}\label{REMIDEMPOTENTSHIFT}

The condition $\delta(\ell)=0$ in Proposition
\ref{POPDP2563t} is needed for the induced map
$\bar{\delta}(s):=\ell\delta(s)\ell$ on
$S:=\ell T\ell$ to be a $\gamma$-derivation. Without this
condition, the map $s\mapsto\ell\delta(s)\ell$ need not be a
derivation of $S$, and consequently the differential polynomial
ring $S[t;\bar{\delta}]$ need not be defined.

Nevertheless, this condition can always be achieved by an inner
shift. Indeed, for an arbitrary $\delta\in\Der_\gamma(T)$, set
\[
r:=\delta(\ell)\ell-\ell\delta(\ell)\in T_\gamma,
\quad
\delta':=\delta-\delta_r.
\]
Since $\ell^2=\ell$, we have
$\delta(\ell)=\delta(\ell)\ell+\ell\delta(\ell)$ and
$\ell\delta(\ell)\ell=0$. It follows that
$\delta_r(\ell)=\delta(\ell)$, hence $\delta'(\ell)=0$.

Therefore Proposition \ref{POPDP2563t} applies to $\delta'$. If
$\bar{\delta}'(s):=\ell\delta'(s)\ell$ for $s\in S$, then
\[
\ell T[t;\delta]\ell
\cong_{\mathrm{gr}}
\ell T[t;\delta']\ell
\cong_{\mathrm{gr}}
S[t;\bar{\delta}'].
\]
Indeed, $\delta'$ differs from $\delta$ by the inner
$\gamma$-derivation $\delta_r$, so Corollary
\ref{CORORE1078} gives $T[t;\delta]\cong_{\mathrm{gr}}T[t;\delta']$; since this isomorphism fixes $T$, it restricts to the corresponding corners. Thus an arbitrary $\gamma$-derivation can always be replaced, up to graded isomorphism, by one satisfying the compatibility condition
required in Proposition \ref{POPDP2563t}.

\end{remark}

\noindent \textnormal{\textbf{Proof of Theorem \ref{THMOR345}.}}
For part (1): by Proposition \ref{PORPOREMORITA123}, gr-simplicity and
gr-primality are invariants of the homogeneous graded Morita equivalence class. Specifically, if $A \approx_{\mathrm{hgr}} R[t;\delta]$ and $R[t;\delta]$ is gr-simple (resp.\ gr-prime), then $A$ is gr-simple (resp.\ gr-prime).

For part (2): let $\ell\in R_e$ be a full idempotent with $\delta(\ell) = 0$. Proposition \ref{POPDP2563t}(2) gives
$\ell R[t;\delta]\ell \cong_{\mathrm{gr}} (\ell R\ell)[t;\bar\delta]$, and Proposition \ref{POPDP2563t}(3) gives
$(\ell R\ell)[t;\bar\delta]\approx_{\mathrm{hgr}} R[t;\delta]$. By \eqref{HG2}, every $A\approx_{\mathrm{hgr}} R[t;\delta]$ arising from a full idempotent $\ell\in R_e$ is graded isomorphic to $\ell R[t;\delta]\ell$, hence to $(\ell R\ell)[t;\bar\delta]$.  \qed

\printbibliography

@article{ait2024,
  author  = {Ait Mohamed, Yassine},
  title   = {On graded rings with homogeneous derivations},
  journal = {Revista de la Uni\'on Matem\'atica Argentina},
  year    = {2024},
  note    = {Published online (early view): December 22, 2024},
  doi     = {10.33044/revuma.4934},
  url     = {https://doi.org/10.33044/revuma.4934}
}

@article{ait2025,
  author     = {Ait Mohamed, Yassine},
  title      = {Generalized Homogeneous Derivations on Graded Rings},
  year       = {2024},
  eprint     = {2412.17187},
  eprinttype = {arXiv},
}

@book{goodearl2004introduction,
  author    = {Goodearl, Kenneth R. and Warfield, Robert B.},
  title     = {An Introduction to Noncommutative {N}oetherian Rings},
  series    = {London Mathematical Society Student Texts},
  volume    = {61},
  publisher = {Cambridge University Press},
  address   = {Cambridge},
  year      = {2004}
}

@article{ore1933theory,
  author    = {Ore, Oystein},
  title     = {Theory of non-commutative polynomials},
  journal   = {Annals of Mathematics},
  volume    = {34},
  number    = {3},
  pages     = {480--508},
  year      = {1933},
  publisher = {JSTOR}
}

@book{rowen1988ring,
  author    = {Rowen, Louis H.},
  title     = {Ring Theory},
  volume    = {I},
  publisher = {Academic Press},
  address   = {Boston},
  year      = {1988}
}

@phdthesis{Jordan1975,
  author       = {Jordan, David A.},
  title        = {Ore Extensions and {J}acobson Rings},
  school       = {University of Leeds},
  year         = {1975},
  type         = {{Ph.D.} Thesis}
}

@incollection{cauchon2006ideaux,
  title={Id{\'e}aux bilat{\`e}res et centre des anneaux de polyn{\^o}mes de Ore sur les anneaux quasi-simples},
  author={Cauchon, G{\'e}rard},
  booktitle={S{\'e}minaire d'Alg{\`e}bre Paul Dubreil: Proceedings, Paris 1977--78 (31{\`e}me Ann{\'e}e)},
  pages={397--407},
  year={2006},
  publisher={Springer}
}

@article{lam2000principal,
  title={Principal one-sided ideals in Ore polynomial rings},
  author={Lam, Tsit Yuen and Leroy, Andr{\'e}},
  journal={Contemporary Mathematics},
  volume={259},
  pages={333--352},
  year={2000},
  publisher={Providence, RI; American Mathematical Society; 1999}
}

@article{krasner1980anneaux,
  title={Anneaux gradu{\'e}s g{\'e}n{\'e}raux},
  author={Krasner, Marc},
  journal={Publications des s{\'e}minaires de math{\'e}matiques et informatique de Rennes},
  number={S3},
  pages={209--308},
  year={1980}
}

@article{jordan1975noetherian,
  title={Noetherian Ore extensions and Jacobson rings},
  author={Jordan, David Alan},
  journal={Journal of the London Mathematical Society},
  volume={2},
  number={3},
  pages={281--291},
  year={1975},
  publisher={Wiley Online Library}
}

@article{oinert2013maximal,
  title={Maximal commutative subrings and simplicity of Ore extensions},
  author={{\"O}inert, Johan and Richter, Johan and Silvestrov, Sergei D},
  journal={Journal of Algebra and its Applications},
  volume={12},
  number={04},
  pages={1250192},
  year={2013},
  publisher={World Scientific}
}

@article{abrams2022morita,
  title={Morita equivalence for graded rings},
  author={Abrams, Gene and Ruiz, Efren and Tomforde, Mark},
  journal={Journal of Algebra},
  volume={617},
  pages={79--112},
  year={2023},
  publisher={Elsevier}
}

@article{GradQ,
  author  = {Kanunnikov, A. L.},
  title   = {Graded Quotient Rings},
  journal = {Journal of Mathematical Sciences},
  volume  = {233},
  number  = {1},
  year    = {2018},
  doi     = {10.1007/s10958-018-3925-7},
  url     = {https://doi.org/10.1007/s10958-018-3925-7}
}

@article{won2018noncommutative,
  title={The noncommutative schemes of generalized Weyl algebras},
  author={Won, Robert},
  journal={Journal of Algebra},
  volume={506},
  pages={322--349},
  year={2018},
  publisher={Elsevier}
}

@article{bavula2020classification,
  title={Classification of Simple Modules of the Ore Extension K [X][Y; fd dX]},
  author={Bavula, Vladimir V},
  journal={Mathematics in Computer Science},
  volume={14},
  number={2},
  pages={317--325},
  year={2020},
  publisher={Springer}
}

@book{Chadeyras1970,
  author    = {Chadeyras, Marcel},
  title     = {Essai d'une th{\'e}orie noeth{\'e}rienne homog{\`e}ne pour les anneaux commutatifs dont la graduation est aussi g{\'e}n{\'e}rale que possible},
  series    = {Bulletin de la Soci{\'e}t{\'e} Math{\'e}matique de France. M{\'e}moire},
  number    = {22},
  publisher = {Soci{\'e}t{\'e} Math{\'e}matique de France},
  year      = {1970},
  doi       = {10.24033/msmf.25}
}

@incollection{IlicGeorgijevic2019,
  author    = {Ili{\'c}-Georgijevi{\'c}, Emil},
  title     = {On radicals of graded ring constructions},
  booktitle = {Rings, Modules and Codes},
  series    = {Contemporary Mathematics},
  volume    = {727},
  pages     = {167--176},
  publisher = {American Mathematical Society},
  address   = {Providence, RI},
  year      = {2019},
  doi       = {10.1090/conm/727/14633}
}

@book{Kelarev2002,
  author    = {Kelarev, Andrei V.},
  title     = {Ring Constructions and Applications},
  series    = {Series in Algebra},
  volume    = {9},
  publisher = {World Scientific},
  year      = {2002},
  isbn      = {9789810247454}
}

@article{Ion1988,
  author  = {Ion, Ion D.},
  title   = {Rings with large centers},
  journal = {Bulletin Math{\'e}matique de la Soci{\'e}t{\'e} des Sciences Math{\'e}matiques de Roumanie},
  volume  = {32},
  number  = {4},
  pages   = {305--310},
  year    = {1988}
}

@misc{Oinert2026central,
  author        = {{\"O}inert, Johan},
  title         = {Central idempotents in group-graded rings},
  year          = {2026},
  eprint        = {2605.20008},
  archiveprefix = {arXiv},
  primaryclass  = {math.RA},
  doi           = {10.48550/arXiv.2605.20008}
}

\end{document}